\documentclass[reqno]{amsart}
\usepackage{amsmath}
\usepackage{mathtools}
\usepackage{amsfonts}
\usepackage{amssymb}
\usepackage{graphicx}
\usepackage{verbatim}
\usepackage{comment}
\usepackage[colorlinks=true]{hyperref}

\usepackage{multicol, color}

\newcounter{mnote}
\newcounter{mnoteE}
\usepackage{graphicx,psfrag}
\numberwithin{equation}{section}
\begin{document}
\title [A determining form for damped driven NLS]{A determining form for the damped driven Nonlinear Schr\"{o}dinger Equation- Fourier modes case}
\date{\today}
%
\author{Michael S. Jolly}
\address[Michael S. Jolly]{Department of Mathematics\\
                Indiana University, Bloomington\\
        Bloomington , IN 47405, USA}
\email[Michael S. Jolly]{msjolly@indiana.edu}
\author{Tural Sadigov}
\address[Tural Sadigov]{Department of Mathematics\\
                Indiana University, Bloomington\\
        Bloomington , IN 47405, USA}
\email[Tural Sadigov]{tsadigov@indiana.edu}
\author{Edriss S. Titi}
\address[Edriss S. Titi]{The Department of Computer Science and Applied Mathematics\\
The Weizmann Institute of Science, Rehovot 76100, Israel. ALSO: Department of Mathematics and
Department of Mechanical and Aerospace Engineering, University of California, Irvine, California, 92697, USA.}
\email[Edriss S. Titi]{etiti@math.uci.edu and edriss.titi@weizmann.ac.il}

\maketitle{}
\begin{abstract}
In this paper we show that the global attractor of the 1D damped, driven, nonlinear Schr\"{o}dinger equation (NLS) is embedded in the long-time dynamics of a determining form. The determining form is an ordinary differential equation in a space of trajectories $X=C_b^1(\mathbb{R}, P_mH^2)$ where $P_m$ is the $L^2$-projector onto the span of the first $m$ Fourier modes. There is a one-to-one identification with the trajectories in the global attractor of the NLS and the steady states of the determining form. We also give an improved estimate for the number of the determining modes.   
\end{abstract}

\section{Introduction}

The damped, driven, nonlinear Schr\"{o}dinger equation (NLS), $\eqref{NLSmain}$,  has been derived in various areas of physics, and widely investigated (see e.g. \cite{Physics} and references therein). In plasma physics, the NLS is a model for the propagation of an intense laser beam through a nonlinear medium (see e.g. \cite{Fibich}). In this model the unknown function $u(x,t)$ is the electrical field amplitude, $t$ is the distance in the direction of the propagation and $x$ is a transverse spatial variable. Absorption of the electromagnetic wave by the medium is accounted for by linear damping.  The NLS also describes the single particle properties of Bose-Einstein condensate (BEC) (see e.g. \cite{BEC}). In case of the BEC, this model is referred to as the Gross-Pitaevski equation. The function $u(x,t)$ describes the macroscopic wave function of the condensate; $t$ is time and $x$ is a spatial variable. A damping term is added to account for inelastic collisions which occur when the particle density is very large. A constant damping rate (absorption) $\gamma$, describes inelastic collisions with the background gas. We note that the NLS is also investigated in deep-water phenomena and in the collapse of Langmuir waves (see e.g. \cite{Fibich}).

The undamped, unforced case has been extensively studied in modern mathematical physics (see e.g. \cite{Bourgain}).  Well-posedness of  $\eqref{NLSmain}$, for nonzero forcing and $\gamma>0$ is established by Ghidaglia in \cite{Gh}, where, under the assumption that the force is either time independent or time periodic, it is also proved that there exists a weak attractor in the Sobolev spaces $H^1$ and $H^2$. Later, it is proved in  \cite{Wang} that this weak attractor is in fact a global attractor in $H^2$ in the strong sense. In \cite{Goubet}, assuming the force is smooth enough and periodic in spatial variable, Goubet proved that the global attractor $\mathcal{A}$ is smooth, meaning it is included and bounded in $H^k$, for any $k\geq 1$. This implies that $\mathcal{A}$ is in $C^\infty$ due to classical Sobolev embeddings theorems. Finally in \cite{OliverTiti}, it is proved that $\mathcal{A}$ is in fact contained in a subclass of the space of real analytical functions provided that the forcing term is real analytic. The long-time dynamics of the damped, driven NLS is entirely contained in the \emph{gobal attractor} $\mathcal{A}$, a compact finite-dimensional set within the infinite-dimensional phase space $H^k$ for any $k\geq 1$ (see \cite{Goubet}). It is shown in \cite{OliverTiti}, for real analytical forcing, that the solutions on the attractor of the NLS are determined uniquely by their nodal values on only two sufficiently close nodes. 

The finite dimensionality for the NLS can be stated more explicitly. It is also known that solutions of the NLS in $\mathcal{A}$ are determined by the asymptotic behavior of a sufficient finite number of Fourier modes (see \cite{Goubet}, \cite{Hale}). To be precise, this means that if two complete trajectories in the global attractor coincide under the projection $P_m$ onto a sufficiently large number, $m$, of low modes, then they are the same trajectory. These $m$-modes are called \emph{determining modes} (see \cite{modes}). This notion of determining modes was used in \cite{Form1} to find a \emph{determining form} for the 2D Navier-Stokes equations (NSE). In \cite{Form1},  the determining form is an ordinary differential equation in an \emph{infinite dimensional} Banach space $X=C_b(\mathbb{R}, P_mH)$, governing the evolution of trajectories. Here $H$ is a Hilbert space which is a natural phase space for the 2D NSE (see  \cite{ConFoias}, \cite{TemamDyn}). The trajectories in the attractor of the 2D NSE are identified with traveling wave solutions of the determining form in  \cite{Form1}. 

A determining form of a different sort was found in  \cite{Form2} for the 2D NSE. It is based on\emph{ data assimilation by feedback control} through a general interpolant operator. It is general in the sense that it can be induced by a variety of determining parameters such as determining modes, nodal values and finite volumes. The steady states of this determining form are precisely the trajectories in the global attractor of the 2D NSE.

Motivation for the determining form comes from the notion of an inertial form. An inertial form for a partial differential equation is an ordinary differential equation restricted to a \emph{finite} dimensional manifold called an \emph{inertial manifold}. We note that it is not known if there is an inertial manifold for 2D NSE. Nor is it known whether there is such a manifold for the NLS. In this paper we adapt the approach in \cite{Form2} for the NLS. While the feedback control approach potentially allows for a variety of interpolant operators, our analysis for the NLS is restricted the case of Fourier modes.  This is done in order to close the {\it a priori} estimates needed in $L^2$, $H^1$, $H^2$, even though there is no dissipative term to absorb the highest derivative. The key step to get a determining form is defining and extending the map $W$ which recovers the high frequency components of a trajectory on the global attractor from the low frequency components. This is done by adding a feedback control term to the NLS (see \cite{AzOT}, \cite{AzT} for feedback controls). The determining form in \cite{Form1} has the map $W$ inserted in the bilinear term of the NSE. The feedback control approach allows us to avoid doing this for the cubic nonlinear term in the NLS. The idea of the feedback control approach is that if we know the $P_m$ projection of the solution of the damped driven NLS on the attractor, we can feed this information into the system to construct the complete solution. It is worth pointing out that this equation is dispersive and merely damped, not strongly dissipative. The analysis used here involves compound functionals motivated by the Hamiltonian structure of the Schr\"{o}dinger equation. 

Section 2 introduces the NLS and special notation. The statements of the main results are mentioned in section  3. A priori estimates are done in section 4. Section 5 contains the main results that we need to obtain a determining form. Section 6 introduces the determining form. Finally, in section  7, we give a different proof of the determining modes property of the NLS through a `reverse' Poincar\'e type inequality. This approach produces an improved estimate for the number of the determining modes for the NLS.

\section{Preliminaries}
We consider the 1D damped, driven, nonlinear Schr\"{o}dinger equation subject to periodic boundary conditions
\begin{align}
&iu_t+u_{xx}+|u|^2u+i\gamma u=f \label{NLSmain} \\
&u(t, x)=u(t, x+L), \qquad \forall \text{ } (t, x)\in \mathbb{R}\times \mathbb{R} \notag \\
&u(0, x)=u_0(x) \notag
\end{align}
where $0<L<\infty$, $0<\gamma$ and $f\neq 0$. We assume that $f$ is time independent, and $f\in L_{per}^2$. Let $0\leq k< \infty$. We denote by $H^k[0, L]$ (or simply $H^k$) Sobolev space of order $k$,
$$H^k[0, L]:=\left \{ f\in L^2[0, L]: \alpha \leq k,  D^\alpha f \text{ exists and } D^\alpha f\in L^2[0, L] \right \},$$
and by $H_{per}^k$, the subspace of $H^k$ consisting of functions which is periodic in $x$, with period $L$. Note that $H_{per}^0[0, L]= L_{per}^2[0, L]$. We assume that $u_0(x)\in H_{per}^2$.  It has been proven in \cite{Gh} that $\eqref{NLSmain}$ has a unique solution $u(x,t)$ such that the mapping 
$$u_0\rightarrow u(t)$$
is continuous on $H^1$, with $u\in L^\infty(\mathbb{R}; H^1)$. The global attractor is the maximal compact invariant set under the solution operator $S(t, \cdot)$. Alternatively, it can be defined  as $\mathcal{A}= \cap_{u_0\in \mathcal{B}}S(t, u_0)$ where $\mathcal{B}$ is an absorbing ball (see e.g. \cite{TemamDyn}). Throughout the paper, we will use the notation 
$$\|u\|^2:=\|u\|_{L^2}^2,$$ 
$$\|u\|_{H^1}^2:=\|u\|_{L^2}^2+\|u_x\|_{L^2}^2,$$ and  
$$\|u\|_{H^2}^2:=\|u\|_{L^2}^2+\|u_{xx}\|_{L^2}^2.$$

To make the flow of the analysis more transparent, we adopt some specialized notation for certain bounding expressions. Bounding expressions that depend on $\gamma$, $f$ (and $\mu$, see $\eqref{main}$) will be denoted by capital letters $\mathcal{R}$ and $K$ with specific indices. The bounding expressions $\mathcal{R}$ with indices $0$, $1$ and $2$  are $L^2$, $H^1$ and $H^2$ bounds, respectively,  for the solution of $\eqref{main}$. Those bounding expressions accented with $\tilde{}$ and $\tilde{\tilde{}}$ will be subsequently improved. As they are improved once, we remove a $\tilde{}$. For example, $\tilde{\tilde{K}}_2$ will be improved once, and we use $\tilde{K}_2$ for the improvement. Then we improve $\tilde{K}_2$ again to get $K_2$ which is the final improvement. Universal constants will be denoted by $c$ and updated throughout the paper. We denote by $P_m$ the $L^2$-projection onto the space $H_m$, where  
\begin{align}
H_m:= \hbox{span}\{e^{ikx\frac{2\pi}{L}}: |k|\leq m\}. \label{Hm}
\end{align}
\section{The Statements Of The Main Results}
We define the following norms,
\begin{align}
&|v|_X= \sup_{s\in \mathbb{R}} \|v(s)\|+ \sup_{s\in \mathbb{R}} \|v_s(s)\| \notag \\
&|v|_{X,0}=\sup_{s\in \mathbb{R}} \|v(s)\|, \label{X0}\\
&|w|_Y= \sup_{s\in \mathbb{R}} \|w(s)\|_{H^2}, \notag
\end{align}
and the following Banach spaces,
\begin{align} 
X=C_{b}^{1}(\mathbb{R}, P_mH^2)=& \{ v:\mathbb{R} \rightarrow  P_mH^2 : v(s) \text{ is } 
\notag \label{X} \\
& \text{ continuous } \forall s\in \mathbb{R} \text{ and }|v|_X<\infty \}, 
\end{align}
\begin{align} 
Y=C_{b}(\mathbb{R}, H^2)= &\{w: \mathbb{R}\rightarrow H^2 :  v(s) \text{ is }\notag \\ 
&\text{ continuous }\forall s\in \mathbb{R} \text{ and } |w|_Y<\infty  \}.\label{Y}
\end{align}
Let $v\in X$, and consider the equation 
\begin{equation}
iw_s+w_{xx}+|w|^2w+i\gamma w=f- i\mu [P_m(w)-v],  \label{main}
\end{equation}
subject to periodic boundary condition
\begin{align*}
w(s, x)=w(s, x+L), \qquad \forall (s, x)\in \mathbb{R}\times \mathbb{R}. 
\end{align*}
We assume that $f\in L_{per}^2$. We first state a new estimate
for the number of determining modes. 
\newtheorem{modes}{Theorem}[section]
\begin{modes}
Assume 
\begin{equation}
m\geq \frac{L}{2\pi}K_{11}-1, \notag \label{numberofmodes}
\end{equation}
 where $K_{11}$ is defined in \eqref{K15}. Then the Fourier projection $P_m$ of $L^2$ onto the space $H_m$, where $H_m$ is defined in $\eqref{Hm}$, is determining for $\eqref{NLSmain}$ i.e, for all $u_1(\cdot), u_2(\cdot)\subset \mathcal{A}$, $P_mu_1(t)=P_mu_2(t), \text{ } \text{for all} \text{ } t\in \mathbb{R}$ implies that $u_1(t)=u_2(t),  \text{ } \text{for all} \text{ } t\in \mathbb{R}$.
\end{modes}
\newtheorem{Rkmode}[modes]{Remark}
\begin{Rkmode} 
By tracking the $\|f\|$ dependence of the bounds throughout the paper, we will show that a sufficient number of determining modes is of order $O(\|f\|^{10})$ as $\|f\| \to \infty$and $O(\gamma^{-12})$ as $\gamma \to 0$.   Following the analysis in Goubet \cite{Goubet}, one can show that a sufficient number of determining modes is of order $O(\gamma^{-12.5})$ as $\gamma \to 0$ and $O(\|f\|^{12})$ as $\|f\| \to \infty$. Thus the functionals in our analysis in $\eqref{phi}$, $\eqref{varphi}$, $\eqref{T}$ and $\eqref{Psi}$, which are naturally motivated by the Hamiltonian structure of the  Schr\"{o}dinger Equation, lead to sharper explicit estimates.  We also mention that the abstract treatment of determining modes by Hale and Raugel is applied to the damped, driven, nonlinear Schr\"odinger equation in \cite{Hale}, but that approach
does not provide estimates for the number of modes needed. 
\end{Rkmode}
The proof of Theorem 3.1 is given in section  7.  It is a byproduct of the proof for the following main result.
\newtheorem{2}[modes]{Theorem} \label{finalresult}
\begin{2}
Let $v\in X$, and $u^*$ be a steady state of equation $\eqref{NLSmain}$. Then we have the following: 
\begin{enumerate}
\item There exists a unique bounded solution $w\in Y$ of $\eqref{main}$, which defines a map $W: X\to Y$, such that $w=W(v)$. 
\item For sufficiently large $m$ and $\mu$, we have $W(P_mu)=u$, for any trajectory $u(s)$, $s\in \mathbb{R}$, in the global attractor of $\eqref{NLSmain}$.  
\item For sufficiently large $m$ and $\mu$, $P_mW: X \to X$ is a locally Lipschitz map.
\item The determining form
\begin{align*}
\frac{dv}{dt}= F(v)= -|v-P_mW(v)|_{X,0}^2 (v- P_mu^*),
\end{align*}
is an ordinary differential equation in a forward invariant set 
$$\{v\in X: |v-P_mu^*|_X< 3(\mathcal{R}_0^0+\mathcal{R'}^0)\},$$ 
and $F$ restricted to that set is globally Lipschitz. Moreover, $P_mu(s)$ is included in that set, for every $u(s)\in \mathcal{A}$. Here $\mathcal{R}_0^0=\mathcal{R}_0|_{\mu=0}$ and $\mathcal{R'}^0=\mathcal{R'}|_{\mu=0}$ are defined in $\eqref{Kinfty}$ and $\eqref{K11}$, respectively.  
\end{enumerate}
\end{2}
Theorem 3.$\ref{finalresult}$ is a combination of results to follow. Item (1) corresponds to Proposition 4.1 and first part of Theorem 5.3, item (2) is equivalent to Proposition 5.1, item (3) is analogous to the second part of Theorem 5.3, and finally item (4) is the summary of Theorem 6.1. The basic idea is to use the Galerkin method to establish a unique bounded solution to $\eqref{main}$, which defines the map $W$. This involves a series of a priori estimates undertaken in the next section. \\\\
Let  $n>m$. Note that $P_nv=v$, for every $v\in X$, and consider a Galerkin approximation of $\eqref{main}$, 
\begin{align}
i\partial_s w_n+\partial_x^2w_n+P_n(|w_n|^2w_n)+i\gamma w_n=P_nf-i\mu(P_mw-v), \label{Galerkin}
\end{align} 
subject to periodic boundary condition, and with the initial data
\begin{align}
w_n(-k,x)=0, \label{ic}
\end{align}
where $w_n\in H_n$, for some $k\in \mathbb{N}$. For simplicity we will drop the subscript $n$. Since $\eqref{Galerkin}$ is an ordinary differential equation with locally Lipschitz nonlinearity, it has a unique, bounded solution $w_n$ on a small interval $[-k, S^*)$, for some $S^*>-k$. We will show that $w_n$ exists globally on the interval $[-k, \infty)$ and is uniformly bounded with respect to $s\in [-k, \infty)$, $n$ and $k$, in the norms of the spaces $L^2$, $H^1$ and $H^2$.
\section{A Priori Estimates} 
\subsection{$L^2$ bound}
Let $[-k, S^*)$ be the maximal interval of existence for $\eqref{Galerkin}$. We will establish here global (in time) uniform in $n$ bounds which will imply, among other things, that $S^*=\infty$. Let us focus below on the interval $[-k, S^*)$. Multiply $\eqref{Galerkin}$ by $\bar{w}$, and integrate
\begin{align}
i\int_0^L w_s\bar{w}+\int_0^Lw_{xx}\bar{w}+\int_0^L|w|^4+&i\int_0^L\gamma |w|^2+i\mu \int_0^LP_m(w)\bar{w} \notag \\ 
&=\int_0^L f\bar{w}+i\mu \int_0^L v\bar{w}. \label{three}
\end{align}
Take the imaginary parts of both sides, and use the fact that $P_m$ is an orthogonal projection and $v\in X$, to get 
$$\frac{1}{2} \frac{d}{ds} \|w\|^2+\gamma \|w\|^2+\mu \|P_m(w)\|^2=Im \int_0^L f\bar{w}+ \mu  Re \int_0^L vP_m\bar{w}.$$
By using H\"{o}lder and Young inequalities, we have
\begin{align*}
\frac{1}{2} \frac{d}{ds} \|w\|^2+\gamma \|w\|^2+\mu \|P_mw\|^2& \leq \|f\| \|w\|+\mu \|v\| \|P_mw\| \notag \\ 
&\leq \frac{\|f\|^2}{2\gamma}+\frac{\gamma \|w\|^2}{2}\notag \\
&+\frac{\mu\|v\|^2}{2}+\frac{\mu\|P_mw\|^2}{2},
\end{align*}
and hence,
$$\frac{d}{ds} \|w\|^2+\gamma \|w\|^2 +\mu \|P_mw\|^2\leq \frac{\|f\|^2}{\gamma}+\mu \|v\|^2 ,$$
for all $s\in [-k, S^*)$. Since $v\in X$, we get
\begin{align}
\frac{d}{ds} \|w\|^2+\gamma \|w\|^2+\mu \|P_mw\|^2\leq \frac{\|f\|^2}{\gamma}+\mu |v|_X^2. \label{l2inequality}
\end{align}
Thus,
$$\frac{d}{ds} \|w\|^2+\gamma \|w\|^2\leq \frac{\|f\|^2}{\gamma}+\mu |v|_X^2.$$
Since, $w(-k,x)=0$, we deduce by Gronwall's lemma that, 
$$\|w(s)\|^2\leq\frac{\|f\|^2}{\gamma^2}+\frac{\mu}{\gamma} |v|_X^2, $$
for all $s\in [-k, S^*)$.  Since the right-hand side is constant, we conclude that $S^*=\infty$, and therefore 
$$\|w(s)\|\leq \frac{\|f\|}{\gamma}+\frac{\mu^{\frac{1}{2}}}{\gamma^{\frac{1}{2}}} |v|_X=: \tilde{\mathcal{R}}_0,$$
for all $s\in [-k, \infty)$, and as a result 
\begin{equation}
\sup_{s\geq -k}\|w(s)\| \leq \tilde{\mathcal{R}}_0. \label{l2bound} 
\end{equation}
Note that the constant $\tilde{\mathcal{R}}_0$ satisfies $\tilde{\mathcal{R}}_0=O(\mu^{\frac{1}{2}})$ as $\mu \to \infty$, and is independent of $k$ and $n$.
\subsection{$H^1$ bound}
Use again the fact that $P_m$ is an orthogonal projection in $\eqref{three}$, and take the real parts of equation $\eqref{three}$:
\begin{equation}
Im \int_0^L w\bar{w}_s= \|w_x\|^2- \|w\|_{L^4}^4+ Re \int_0^L f\bar{w}- \mu Im \int_0^L v\bar{w}, \label{IM}
\end{equation}
for all $s\in [-k, \infty)$. Now, multiply $\eqref{Galerkin}$ by $\bar{w}_s$ and integrate with respect to $x$ over $[0,L]$ to obtain
\begin{align*}
i\|w_s\|^2- \int_0^L w_x \bar{w}_{xs} + \int_0^L |w|^2w\bar{w}_s+&i\gamma \int_0^L w\bar{w}_s+i\mu\int_0^L P_m(w)\bar{w}_s\\
&=\int_0^L f \bar{w}_s+ i\mu\int_0^L v \bar{w}_s.
\end{align*}
Take the real part of the above equation to obtain
\begin{align}
\frac{d}{ds}\|w_x\|^2-\frac{1}{2}\frac{d}{ds}\|w\|_{L^4}^4+ 2\gamma Im \int_0^L w\bar{w}_s&+2\mu Im \int_0^L P_mwP_m\bar{w}_s=\notag \\
&-2 Re \int_0^L f \bar{w}_s+ 2\mu Im \int_0^L v \bar{w}_s.  \label{six} 
\end{align}
\\\\
To eliminate the third term in $\eqref{six}$, we will use $\eqref{IM}$. For the fourth term, we take the $P_m$ projection of equation $\eqref{main}$, then multiply by $P_m\bar{w}$, integrate, and take the real parts to find
\begin{align}
Im \int_0^L P_mwP_m\bar{w}_s- \|P_mw_x\|^2+& Re \int_0^L P_m(|w|^2w)P_m\bar{w}\notag \\
&= Re \int_0^L P_mfP_m\bar{w}- \mu Im \int_0^L  vP_m\bar{w}. \label{IMpr}
\end{align}
Now combine as follows: $(-2\gamma)\times \eqref{IM}+(-2\mu)\times \eqref{IMpr}+ \eqref{six}$ to get 
\begin{align} 
\frac{d\phi}{ds}+4\gamma \phi=& 2\gamma \|w_x\|^2+6\gamma Re\int_0^L f\bar{w} -6\mu \gamma Im\int_0^L v\bar{w} \notag \\
&+2\mu Re \int_0^L P_m(|w|^2w)P_m\bar{w}- 2\mu \|P_mw_x\|^2 \notag \\
&-2\mu Re \int_0^L P_mfP_m\bar{w}+2\mu^2 Im \int_0^L  vP_m\bar{w} \notag\\
&-2\mu Im \int_0^L  v_s \bar{w}, \label{phiequation}
\end{align}
where
\begin{align}
\phi(s)=\|w_x\|^2-\frac{1}{2}\|w\|_{L^4}^4+2Re \int_0^L f \bar{w}-2\mu Im \int_0^L  v \bar{w}. \label{phi}
\end{align}
Since $w(-k,x)=0$ for all $x$, we have that $w_x(-k,x)=0$ for all $x$, and thus $\phi(-k)=0$. \\
We estimate the right hand side of $\eqref{phiequation}$ using the H\"{o}lder, Young and Agmon inequalities and $\eqref{l2bound}$ as follows
$$2\gamma \|w_x\|^2\leq 2\gamma \|w\|_{H^1}^2,$$ 
$$6\gamma Re\int_0^L f\bar{w} \leq 6\gamma \|f\|\|w\| \leq 6\gamma \|f\| \tilde{\mathcal{R}}_0,$$
$$-6\mu \gamma Im\int_0^L  v\bar{w}\leq 6\mu \gamma  \|v\| \|w\|\leq 6\mu \gamma  |v|_X \tilde{\mathcal{R}}_0,$$
\begin{align*}
2\mu Re \int_0^L P_m(|w|^2w)P_m\bar{w}&= 2\mu Re \int_0^L |w|^2wP_m\bar{w}\\
&\leq 2\mu \|w\|_\infty^2 \|w\| \|P_mw\| \\
&\leq 2\mu (c\|w\| \|w\|_{H^1}) \|w\| \|P_mw\| \\
&\leq 2\mu c\|w\|^3 \|w\|_{H^1}  \\
&\leq \frac{\mu^2c^2 \|w\|^6}{\gamma }+ \gamma  \|w\|_{H^1}^2 \\
&\leq \frac{\mu^2 c^2\tilde{\mathcal{R}}_0^6}{\gamma }+ \gamma  \|w\|_{H^1}^2.
\end{align*}
Moreover, we have
$$-2\mu Re \int_0^L P_mfP_m\bar{w}\leq 2\mu \|P_mf\| \|P_mw\| \leq 2\mu \|f\| \|w\| \leq 2\mu \|f\| \tilde{\mathcal{R}}_0,$$
$$2\mu^2 Im \int_0^L  vP_m\bar{w} \leq 2\mu^2  \|v\| \|P_mw\| \leq 2\mu^2  |v|_X \tilde{\mathcal{R}}_0,$$
$$-2\mu Im \int_0^L  v_s \bar{w}\leq 2\mu  \|v_s\| \|w\|\leq 2\mu  |v|_X\tilde{\mathcal{R}}_0.$$
Putting together the above estimates, we obtain 
\begin{equation}
\frac{d\phi}{ds}+4\gamma \phi \leq 3\gamma \|w\|_{H^1}^2+ \tilde{\tilde{K}}_1, \label{sekiz}
\end{equation}
where
\begin{align*} 
\tilde{\tilde{K}}_1:=  6\gamma \|f\| \tilde{\mathcal{R}}_0+ 6\mu \gamma |v|_X \tilde{\mathcal{R}}_0+ \frac{\mu^2 c^2\tilde{\mathcal{R}}_0^6}{\gamma }+2\mu \|f\| \tilde{\mathcal{R}}_0+2\mu^2 |v|_X \tilde{\mathcal{R}}_0+2\mu |v|_X\tilde{\mathcal{R}}_0.
\end{align*}
Then by using the Agmon, H\"{o}lder, and Young inequalities in $\eqref{phi}$, we get
\begin{align}
\phi(s)&\geq \|w_x\|^2- \frac{1}{2}c \|w\|^3 \|w\|_{H^1}-2\|f\| \|w\|-2\mu  \|v\| \|w\| \notag \\
&=  \|w_x\|^2- \frac{1}{2}c \|w\|^3 \|w_x\|-\frac{1}{2} c\|w\|^4-2\|f\| \|w\|-2\mu  \|v\| \|w\| \notag \\
&\geq \|w_x\|^2- \frac{c^2\|w\|^6}{16\xi}-\xi \|w_x\|^2-\frac{1}{2}c \|w\|^4-2\|f\| \tilde{\mathcal{R}}_0-2\mu  |v|_X \tilde{\mathcal{R}}_0 \notag \\
&\geq (1-\xi)\|w_x\|^2+ (1-\xi)\|w\|^2 \notag \\ 
&\text{	}- \left [\frac{c^2\tilde{\mathcal{R}}_0^6}{16\xi}+\frac{1}{2}c \tilde{\mathcal{R}}_0^4+2\|f\| \tilde{\mathcal{R}}_0+2\mu  |v|_X \tilde{\mathcal{R}}_0+(1-\xi)\tilde{\mathcal{R}}_0^2\right ], \label{phiestimation}
\end{align}
where $0<\xi<1$, $\xi$ to be chosen later.
Thus, we have
$$\phi\geq (1-\xi)\|w\|_{H^1}^2-\tilde{\tilde{K}}_2,$$
where
$$\tilde{\tilde{K}}_2:=\frac{c^2\tilde{\mathcal{R}}_0^6}{16\xi}+\frac{1}{2} c\tilde{\mathcal{R}}_0^4+2\|f\| \tilde{\mathcal{R}}_0+2\mu |v|_X \tilde{\mathcal{R}}_0+(1-\xi)\tilde{\mathcal{R}}_0^2 ,$$
and hence 
\begin{equation}
\|w\|_{H^1}^2\leq \frac{1}{1-\xi}\phi + \frac{\tilde{\tilde{K}}_2}{1-\xi}. \label{dokuz}
\end{equation}
Use $\eqref{dokuz}$ in $\eqref{sekiz}$ to obtain
$$\frac{d\phi}{ds}+\frac{1-4\xi}{1-\xi}\gamma \phi \leq \tilde{\tilde{K}}_3,$$
where 
\begin{align*}
\tilde{\tilde{K}}_3:= \frac{3\gamma \tilde{\tilde{K}}_2}{1-\xi}+\tilde{\tilde{K}}_1.
\end{align*}
Choose $\xi=\frac{1}{7}$, so that 
\begin{align*}
\frac{1-4\xi}{1-\xi}= \frac{1}{2}. 
\end{align*}
So we have 
$$\frac{d\phi}{ds}+\frac{\gamma}{2}\phi \leq \tilde{\tilde{K}}_3.$$
Since $\phi(-k)=0$ , applying the Gronwall lemma, we have
$$\phi(w(s))\leq \tilde{\tilde{K}}_4,$$
for all $s\geq -k$, where $\tilde{\tilde{K}}_4:= 2\tilde{\tilde{K}}_3/\gamma$. From $\eqref{dokuz}$, we obtain 
$$\|w(s)\|_{H^1}^2\leq \frac{7}{6}\phi+\frac{7}{6} \tilde{\tilde{K}}_2\leq \frac{7}{6} (\tilde{\tilde{K}}_4+\tilde{\tilde{K}}_2),$$
for all $s\geq -k$. Therefore, we have $\tilde{\tilde{K}}_1=O(\mu^5), \tilde{\tilde{K}}_2=O(\mu^3), \tilde{\tilde{K}}_3=O(\mu^5), \tilde{\tilde{K}}_4=O(\mu^5)$ as $\mu \to \infty$.
Thus 
\begin{align}
\sup_{s\geq -k} \|w(s)\|_{H^1} \leq \tilde{\tilde{\mathcal{R}}}_1:= \sqrt{\frac{7}{6} (\tilde{\tilde{K}}_4+\tilde{\tilde{K}}_2)}=\sqrt{\frac{28}{3}\tilde{\tilde{K}}_2+\frac{7}{3\gamma}\tilde{\tilde{K}}_1}=O(\mu^\frac{5}{2}), \label{K61}
\end{align}
as $\mu \to \infty$. Note that $\tilde{\tilde{\mathcal{R}}}_1$ is independent of  $k$ and $n$. 
\subsection{Improved $L^2$ bound}
We now use the $H^1$-bound in $\eqref{K61}$ to obtain a better $L^2$-bound. We rewrite $\eqref{l2inequality}$ as
$$\frac{d}{ds} \|w\|^2+\gamma \|w\|^2+\mu \|w\|^2-\mu\|Q_mw\|^2\leq \frac{\|f\|^2}{\gamma}+\mu |v|_X^2,$$
where $Q_m= I- P_m$. Thus,
$$\frac{d}{ds} \|w\|^2+(\gamma+\mu )\|w\|^2\leq \frac{\|f\|^2}{\gamma}+\mu |v|_X^2+ \mu\|Q_mw\|^2 .$$
By the generalized Poincar\'e inequality we have 
\begin{align*} 
\|Q_mw\|^2\leq \frac{L^2}{4\pi^2}\frac{1}{(m+1)^2}\|w_x\|^2\leq\frac{L^2}{4\pi^2}\frac{\tilde{\tilde{\mathcal{R}}}_1^2}{(m+1)^2}. 
\end{align*}
If we choose $m$ large enough such that 
\begin{align}
\frac{\tilde{\tilde{\mathcal{R}}}_1^2}{(m+1)^2}\frac{L^2}{4\pi^2}<1, \label{E1} 
\end{align}
then
$$\frac{d}{ds} \|w\|^2+(\gamma+\mu )\|w\|^2\leq \frac{\|f\|^2}{\gamma}+\mu |v|_X^2+ \mu.$$
Now, we apply the Gronwall Lemma, using the fact that $\|w(-k)\|=0$, to obtain
\begin{align*}
\|w(s)\|^2\leq  \frac{\|f\|^2}{\gamma(\gamma+\mu)}+\frac{\mu |v|_X^2}{\gamma+\mu}+ \frac{\mu}{\gamma+\mu},
\end{align*}
for every $s\geq-k$, and hence, 
$$\sup_{s\geq-k}\|w(s)\|^2\leq   \frac{\|f\|^2}{\gamma(\gamma+\mu)}+\frac{\mu |v|_X^2}{\gamma+\mu}+ \frac{\mu}{\gamma+\mu}.$$
As a result, 
\begin{align*}
\sup_{s\geq-k}\|w(s)\|\leq \mathcal{R}_0, 
\end{align*}
where 
\begin{align}
\mathcal{R}_0:=\frac{\|f\|}{\sqrt{\gamma(\gamma+\mu)}}+\sqrt{\frac{\mu}{\gamma+\mu}}|v|_X+ \sqrt{\frac{\mu}{\gamma+\mu}}=O(\mu^0), \label{K1}
\end{align}
as $\mu \to \infty$. Note that $\mathcal{R}_0$ depends neither on $k$, nor on $n$. So by choosing $m$ large enough satisfying $\ref{E1}$, we get an $L^2$-bound which is uniform in $\mu$. Inserting $\mathcal{R}_0$ in place of $\tilde{\mathcal{R}}_0$ in the proof of the $H^1$-bound, yields new constants 
$$\tilde{K}_1=O(\mu^2), \tilde{K}_2=O(\mu), \tilde{K}_3=O(\mu^2),  \tilde{K}_4=O(\mu^2),$$ 
replacing $\tilde{\tilde{K}}_1$, $\tilde{\tilde{K}}_2$, $\tilde{\tilde{K}}_3$ and $\tilde{\tilde{K}}_4$ , respectively as $\mu \to \infty$. As as result,
$$\sup_{s\geq -k}\|w(s)\|_{H^1}\leq \tilde{\mathcal{R}}_1:=\sqrt{\frac{28}{3}\tilde{K}_2+\frac{7}{3\gamma}\tilde{K}_1}=O(\mu),$$
as $\mu \to \infty$, where  
\begin{align*} 
\tilde{K}_1:=  6\gamma \|f\| \mathcal{R}_0+ 6\mu \gamma |v|_X \mathcal{R}_0+ \frac{\mu^2 c^2\mathcal{R}_0^6}{\gamma }+2\mu \|f\| \mathcal{R}_0+2\mu^2 |v|_X \mathcal{R}_0+2\mu |v|_X\mathcal{R}_0, 
\end{align*}
and 
$$\tilde{K}_2:=\frac{c^2\mathcal{R}_0^6}{16\xi}+\frac{1}{2} c\mathcal{R}_0^4+2\|f\|\mathcal{R}_0+2\mu |v|_X \mathcal{R}_0+(1-\xi)\mathcal{R}_0^2.$$
\subsection{$H^2$ bound}
Multiply $\eqref{Galerkin}$ with $\bar{w}_{xxs}+\gamma \bar{w}_{xx}$, integrate, and take the real parts:
\begin{align*}
Re \int_0^L i w_s\bar{w}_{xxs}+&  \gamma Re\int_0^L i w_s \bar{w}_{xx}+ Re \int_0^L w_{xx}\bar{w}_{xxs}+ \gamma Re\int_0^L |w_{xx}|^2\\
+&Re\int_0^L |w|^2w\bar{w}_{xxs}+ \gamma Re\int_0^L |w|^2w\bar{w}_{xx}+ \gamma Re\int_0^L i w\bar{w}_{xxs}\\
+&\gamma^2 Re\int_0^L i w\bar{w}_{xx}+\mu Re\int_0^L i P_mw P_m\bar{w}_{xxs}+ \mu \gamma Re\int_0^L i P_mwP_m\bar{w}_{xx}\\
= &Re\int_0^L f\bar{w}_{xxs}+ \gamma Re\int_0^L f\bar{w}_{xx}+\mu Re\int_0^L i  v\bar{w}_{xxs}+\mu \gamma Re\int_0^L i  v \bar{w}_{xx}.
\end{align*}
We now estimate term by term, using integration by parts in most cases.
$$Re \int_0^L i w_s\bar{w}_{xxs}=- Re\int_0^L i |w_{xs}|^2=0,$$
$$\gamma Re\int_0^L i w_s \bar{w}_{xx}= - \gamma Re \int_0^L i w_{xs}\bar{w}_x=\gamma Im\int_0^L w_{xs}\bar{w}_x=-\gamma Im\int_0^L w_x\bar{w}_{xs},$$
$$Re \int_0^L w_{xx}\bar{w}_{xxs}=\frac{1}{2}\frac{d}{ds}\|w_{xx}\|^2,$$
$$\gamma Re\int_0^L |w_{xx}|^2= \gamma \|w_{xx}\|^2,$$
\begin{align*}
Re\int_0^L |w|^2w\bar{w}_{xxs}=- Re \int_0^L (w^2\bar{w})_x\bar{w}_{xs}=&- Re\int_0^L 2ww_x\bar{w}\bar{w}_{xs}- Re\int_0^L w^2\bar{w}_x\bar{w}_{xs}\\
= &-\int_0^L |w|^2\frac{d}{ds}|w_x|^2-\frac{1}{2} Re\int_0^L w^2\frac{d}{ds}\bar{w}_x^2,
\end{align*}
\begin{align*}
\gamma Re\int_0^L |w|^2w\bar{w}_{xx}=- \gamma Re \int_0^L (w^2\bar{w})_x\bar{w}_x=&- \gamma Re\int_0^L 2ww_x\bar{w}\bar{w_x}- \gamma Re\int_0^L w^2\bar{w}_x\bar{w}_x\\
=&-2\gamma \int_0^L |w|^2|w_x|^2-\gamma Re \int_0^L w^2\bar{w}_x^2,
\end{align*}
$$\gamma Re\int_0^L i w\bar{w}_{xxs}=- \gamma Re\int_0^L i w_x\bar{w}_{xs}=\gamma Im\int_0^L w_x\bar{w}_{xs},$$
$$\gamma^2 Re\int_0^L i w\bar{w}_{xx}=-\gamma^2 Re\int_0^L i |w_x|^2=0,$$
$$\mu Re\int_0^L i P_mw P_m\bar{w}_{xxs}=-\mu Re\int_0^L i P_mw_xP_m\bar{w}_{xs}= \mu Im\int_0^L P_mw_xP_m\bar{w}_{xs},$$
$$\mu \gamma Re\int_0^L i P_mwP_m\bar{w}_{xx}= -\mu \gamma Re\int_0^L i |P_mw_x|^2=0.$$
Now, we combine the above terms to get
\begin{align}
\frac{1}{2}\frac{d\varphi}{ds}+\gamma \varphi=& -\int_0^L 2Re(w\bar{w}_s|w_x|^2)-Re\int_0^L ww_s\bar{w}_x^2 \notag \\
&+\mu Im\int_0^L  v_s\bar{w}_{xx}+\gamma \mu Im\int_0^L  v\bar{w}_{xx}\notag \\
&- \gamma Re\int_0^L f\bar{w}_{xx}- \mu Im\int_0^L P_mw_xP_m\bar{w}_{xs} ,\label{on}
\end{align}
for all $s\geq -k$, where 
\begin{align}
\varphi(w):= \|w_{xx}\|^2&- 2\int_0^L |w|^2|w_x|^2- Re\int_0^L w^2\bar{w}_x^2 \notag \\
&-2 Re\int_0^L f\bar{w}_{xx}+2\mu Im\int_0^L  v\bar{w}_{xx}. \label{varphi}
\end{align}
Observe again that since $w(-k,x)=w_x(-k,x)=w_{xx}(-k,x)=0$, for all $x\in [0, L]$, we have $\varphi(-k)=0$. We write 
\begin{equation}
w_s=i w_{xx}+ h, \label{h}
\end{equation}
where $h:= i|w|^2w-\gamma w-\mu P_mw+\mu  v-i f$. Observe, thanks to Agmon's inequality, that 
\begin{align*}
\|h(s)\|&\leq c\|w\|^2\|w\|_{H^1}+(\gamma+\mu) \|w\|+\mu  \|v\|+\|f\|\\
&\leq c\mathcal{R}_0^2\tilde{\mathcal{R}}_1+ (\gamma+\mu) \mathcal{R}_0+ \mu   |v|_X+\|f\|,
\end{align*}
for all $s\geq -k$. We estimate each term on the right-hand side of $\eqref{on}$. We use $\eqref{h}$, as well as the Young, H\"{o}lder,   and Agmon inequalities to obtain\\
\begin{align*}
-\int_0^L 2Re(w\bar{w}_s|w_x|^2)&\leq 2\int_0^L |w||w_s||w_x|^2= 2\int_0^L |w| |i w_{xx}+ h | |w_x|^2\\
&\leq 2\int_0^L |w| |w_{xx}| |w_x|^2+2\int_0^L |w| |h| |w_x|^2\\
&\leq 2\|w\|_\infty \|w_x\|_\infty \int_0^L |w_{xx}||w_x|+2\|w\| \|h\| \|w_x\|_{\infty}^2\\
&\leq 2\|w\|_\infty \|w_x\|_\infty  \|w_{xx}\| \|w_x\|+ 2\|w\| \|h\| \|w_x\|_{\infty}^2\\
&\leq c\|w\|_\infty \|w_x\|^{\frac{3}{2}} \|w_{xx}\|^{\frac{3}{2}}+ c\|w\| \|h\| \|w_x\| \|w_{xx}\|\\
&\leq c(\|w\|^{\frac{1}{2}}\|w\|_{H^1}^2)\|w_{xx}\|^{\frac{3}{2}}+ c(\|w\| \|h\| \|w_x\|)\|w_{xx}\|\\
&\leq c(\mathcal{R}_0^{\frac{1}{2}}\tilde{\mathcal{R}}_1^2 )\|w_{xx}\|^{\frac{3}{2}}\\
&+ c\mathcal{R}_0\tilde{\mathcal{R}}_1(\mathcal{R}_0^2\tilde{\mathcal{R}}_1+ (\gamma+\mu) \mathcal{R}_0+ \mu   |v|_X+\|f\|)\|w_{xx}\|\\
&\leq \frac{\gamma}{20}\|w_{xx}\|^2+ \frac{c(\mathcal{R}_0^{\frac{1}{2}}\tilde{\mathcal{R}}_1^2 )^4}{\gamma^3}\\
&+ \frac{c(\mathcal{R}_0\tilde{\mathcal{R}}_1(\mathcal{R}_0^2\tilde{\mathcal{R}}_1+ (\gamma+\mu) \mathcal{R}_0+ \mu   |v|_X+\|f\|))^2}{\gamma}\\
&= \frac{\gamma}{20}\|w_{xx}\|^2+\tilde{K}_5, 
\end{align*}
where 
$$\tilde{K}_5:= \frac{c(\mathcal{R}_0^{\frac{1}{2}}\tilde{\mathcal{R}}_1^2 )^4}{\gamma^3}+\frac{c\{\mathcal{R}_0\tilde{\mathcal{R}}_1[\mathcal{R}_0^2\tilde{\mathcal{R}}_1+ (\gamma+\mu) \mathcal{R}_0+ \mu   |v|_X+\|f\|]\}^2}{\gamma}.$$
Similarly, we have
$$-Re\int_0^L ww_s\bar{w}_x^2\leq \frac{\gamma}{10}\|w_{xx}\|^2+ \tilde{K}_5,$$
and
$$\mu Im\int_0^L  v_s\bar{w}_{xx}\leq \mu \|v_{s}\|\|w_xx\|\leq \frac{\gamma}{20}\|w_{xx}\|^2+\frac{c\mu^2|v|_X^2}{\gamma},$$ 
$$ \gamma \mu Im\int_0^L  v\bar{w}_{xx}\leq \gamma \mu \|v\|\|w_{xx}\|\leq \frac{\gamma}{20}\|w_{xx}\|^2+c\gamma \mu^2|v|_X^2 ,$$
$$- \gamma Re\int_0^L f\bar{w}_{xx}\leq \gamma \|f\| \|w_{xx}\|\leq \frac{\gamma}{20}\|w_{xx}\|^2+c\gamma \|f\|^2.$$
For the term $-\mu Im\int_0^L P_mw_xP_m\bar{w}_{xs}$, we take the $P_m$ projection of equation $\eqref{Galerkin}$, multiply it with $\mu P_m\bar{w}_{xx}$, integrate, and take the real part to get 
\begin{align*}
-\mu Im\int_0^L P_mw_xP_m\bar{w}_{xs}=& -\mu \|P_mw_{xx}\|^2-\mu Re\int_0^L P_m(|w|^2w)P_m\bar{w}_{xx}\\
&+\mu Re\int_0^L fP_m\bar{w}_{xx}-\mu^2 Im\int_0^L  vP_m\bar{w}_{xx} \\
\leq& -\mu \|P_mw_{xx}\|^2+ \mu \|w\|_\infty^2\|w\|\|P_mw_{xx}\|+\\
&+\mu \|f\|\|P_mw_{xx}\|+ \mu^2 \|v\|\|P_mw_{xx}\|.
\end{align*}
Apply Young's inequality to eliminate $\mu \|P_mw_{xx}\|^2$, and then Agmon inequality, to get
$$-\mu Im\int_0^L P_mw_xP_m\bar{w}_{xs}\leq c\mu(\mathcal{R}_0^2\tilde{\mathcal{R}}_1)^2+c\mu \|f\|^2+c\mu^3|v|_X^2.$$
Combine the above terms to obtain
\begin{equation}
\frac{1}{2}\frac{d\varphi}{ds}+\gamma \varphi\leq \tilde{K}_6+ \frac{\gamma}{4}\|w_{xx}\|^2 \label{oniki},
\end{equation}
where 
\begin{align*}
\tilde{K}_6:=&   2\tilde{K}_5+ \frac{c\mu^2|v|_X^2}{\gamma}+ c\gamma \mu^2|v|_X^2 +c\gamma \|f\|^2+\\
&+c\mu(\mathcal{R}_0^2\tilde{\mathcal{R}}_1)^2+c\mu \|f\|^2+c\mu^3|v|_X^2.
\end{align*} 
Note that $\tilde{K}_6=O(\mu^8)$ as $\mu \to \infty$. From $\eqref{varphi}$ we obtain
\begin{equation}
\varphi \geq \frac{1}{2}\|w_{xx}\|^2- \tilde{K}_7 \label{onuc},
\end{equation}
where $\tilde{K}_7:= c(\mathcal{R}_0\tilde{\mathcal{R}}_1^3+\|f\|^2+\mu^2|v|_X^2)$. Use $\eqref{onuc}$ in $\eqref{oniki}$, to get
\begin{align*}
\frac{1}{2}\frac{d\varphi}{ds}+\gamma \varphi&\leq \tilde{K}_6+ \frac{\gamma}{2}(\varphi +\tilde{K}_7)\\
&= \frac{\gamma}{2}\varphi +(\frac{\gamma}{2}\tilde{K}_7+\tilde{K}_6)\\
&= \frac{\gamma}{2}\varphi+(\frac{\gamma}{2}\tilde{K}_7+\tilde{K}_6).
\end{align*}
Thus we have
$$\frac{d\varphi}{ds}+\gamma \varphi\leq \gamma\tilde{K}_7+2\tilde{K}_6.$$
Since $\varphi(-k)=0$,  we have, thanks to Gronwall's lemma,  $\varphi(s)\leq \tilde{K}_7+\frac{2\tilde{K}_6}{\gamma}$, for all $s\geq -k$. From $\eqref{onuc}$, we get
$$ \|w_{xx}(s)\|^2\leq 2\varphi+2\tilde{K}_7\leq \frac{4\tilde{K}_6}{\gamma}+ 4\tilde{K}_7,$$
for all $s\geq -k$. Thus we get
\begin{align*}
\sup_{s\geq -k} \|w(s)\|_{H^2} \leq \tilde{\mathcal{R}}_2, 
\end{align*}
where 
\begin{align}
\tilde{\mathcal{R}}_2:=\sqrt{\frac{4\tilde{K}_6}{\gamma}+ 4\tilde{K}_7}+\mathcal{R}_0=O(\mu^4), \label{K71} 
\end{align}
as $\mu \to \infty$. Comparing to $\eqref{K61}$, we observe that $\tilde{\mathcal{R}}_2 >> \tilde{\tilde{\mathcal{R}}}_1$, for large $\mu$. 
\subsection{Improved $H^1$ bound}
We now use the $H^2$-bound to obtain a better $H^1$-bound. From $\eqref{phiequation}$ and $\eqref{phi}$, we realize that
\begin{align} 
\frac{d\phi}{ds}+4\gamma \phi+\mu \phi=& 2\gamma \|w_x\|^2+6\gamma Re\int_0^L f\bar{w} -6\mu \gamma Im\int_0^L v\bar{w}\notag \\
&+2\mu Re \int_0^L P_m(|w|^2w)P_m\bar{w}- 2\mu \|P_mw_x\|^2\notag \\
&-2\mu Re \int_0^L P_mfP_m\bar{w}-2\mu Im \int_0^L  v_s \bar{w}\notag \\
&+\mu \|w_x\|^2-\frac{\mu}{2}\|w\|_{L^4}^4+2\mu Re \int_0^L f\bar{w}. \label{phiequation2}
\end{align}
We estimate the right-hand side of $\eqref{phiequation2}$ as follows
$$2\gamma \|w_x\|^2\leq 2\gamma \|w\|_{H^1}^2,$$ 
$$6\gamma Re\int_0^L f\bar{w} \leq 6\gamma \|f\|\|w\| \leq 6\gamma \|f\| \mathcal{R}_0,$$
$$-6\mu \gamma Im\int_0^L  v\bar{w}\leq 6\mu \gamma  \|v\| \|w\|\leq 6\mu \gamma  |v|_X \mathcal{R}_0,$$
\begin{align*}
2\mu Re \int_0^L P_m(|w|^2w)P_m\bar{w}&= 2\mu Re \int_0^L |w|^2wP_m\bar{w}\\
&\leq 2\mu \|w\|_\infty^2 \|w\| \|P_mw\| \\
&\leq 2\mu c (\|w\| \|w\|_{H^1}) \|w\| \|P_mw\| \\
&\leq 2\mu c \|w\|^3 \|w\|_{H^1}  \\
&\leq \mu c^2\|w\|^6+ \mu \|w\|_{H^1}^2 \\
&\leq \mu (c^2\|w\|^6+\|w\|^2)+ \mu \|w_x\|^2 \\
&\leq \mu (c^2\mathcal{R}_0^6+\mathcal{R}_0^2)+ \mu \|w_x\|^2.
\end{align*}
Choose $m$ large enough so that 
\begin{align}
\frac{\tilde{\mathcal{R}}_2^2L^2}{4\pi^2(m+1)^2}\leq 1. \label{E2}
\end{align}
Then, 
\begin{align*}
- 2\mu \|P_mw_x\|^2&=  -2\mu \|w_x\|^2+ 2\mu \|Q_mw_x\|^2\\
&\leq -2\mu \|w_x\|^2+\frac{2\mu}{((m+1)\frac{2\pi}{L})^2}\|w_{xx}\|^2\\
&\leq -2\mu \|w_x\|^2+2\mu,
\end{align*}
\begin{align*}
2\mu Re \int_0^L f\bar{w}-2\mu Re \int_0^L P_mfP_m\bar{w}&\leq 2\mu \|f\| \|Q_mw\|\\
& \leq 2\mu \|f\| \|w\|\\
&\leq 2\mu \|f\| \mathcal{R}_0,
\end{align*} 
$$-2\mu Im \int_0^L  v_s \bar{w}\leq 2\mu  \|v_s\| \|w\|\leq 2\mu  |v|_X\mathcal{R}_0.$$ 
Add the above terms to  obtain 
\begin{align}
\frac{d\phi}{ds}+4\gamma \phi+ \mu \phi \leq 2\gamma \|w\|_{H^1}^2 + K_1, \label{phiinequality}
\end{align}
where
\begin{align*} 
K_1:=  6\gamma \|f\| \mathcal{R}_0&+ 6\mu \gamma |v|_X \mathcal{R}_0+\mu(c^2\mathcal{R}_0^6+\mathcal{R}_0^2)+2\mu+2\mu \|f\| \mathcal{R}_0+2\mu |v|_X\mathcal{R}_0.
\end{align*}
Now, as in $\eqref{phiestimation}$
\begin{align*}
\phi(s)&\geq \|w_x\|^2- \frac{3c^2\|w\|^6}{16}-\frac{1}{3} \|w_x\|^2-\frac{1}{2} c\|w\|^4-2\|f\| \mathcal{R}_0-2\mu  |v|_X \mathcal{R}_0\\
&\geq \frac{2}{3}\|w_x\|^2+ \frac{2}{3}\|w\|^2- \left [\frac{3c^2\mathcal{R}_0^6}{16}+\frac{1}{2}c \mathcal{R}_0^4+2\|f\| \mathcal{R}_0+2\mu  |v|_X \mathcal{R}_0+\frac{2}{3}\mathcal{R}_0^2\right ] \\
&=\frac{2}{3}\|w\|_{H^1}^2-K_2,
\end{align*}
where 
$$K_2:= \frac{3c^2\mathcal{R}_0^6}{16}+\frac{1}{2} c\mathcal{R}_0^4+2\|f\| \mathcal{R}_0+2\mu  |v|_X \mathcal{R}_0+\frac{2}{3}\mathcal{R}_0^2.$$
Thus $\|w\|_{H^1}^2\leq \frac{3}{2}\phi(w)+\frac{3}{2}K_2$. Use this in $\eqref{phiinequality}$ to obtain
\begin{align*}
\frac{d\phi}{ds}+ \gamma \phi +\mu \phi \leq  3\gamma K_2+K_1
\end{align*}
Thus, since $\phi(-k)=0$, by virtue of Gronwall Lemma, we have  
$$\phi(s)\leq \frac{3\gamma K_2+K_1}{\gamma+\mu},$$
for all $s\geq -k$. Thus,
\begin{align}
\sup_{s\geq -k}\|w(s)\|_{H^1}\leq \mathcal{R}_1:=\sqrt{\frac{(\frac{3}{2}\mu+6\gamma)K_2+K_1}{\gamma+\mu}}=O(\mu^{\frac{1}{2}}),\label{K6}
\end{align}
as $\mu \to \infty$. Inserting $\mathcal{R}_1$ in place of $\tilde{\mathcal{R}}_1$ in the proof of the $H^2$-bound, yields new $H^2$-bound 
\begin{align}
\sup_{s\geq -k}\|w(s)\|_{H^2}\leq \mathcal{R}_2=\sqrt{\frac{4K_6}{\gamma}+ 4K_7}+\mathcal{R}_0=O(\mu^2) \label{K10},
\end{align}
as $\mu \to \infty$, where 
\begin{align*}
K_6:=&   2K_5+ \frac{c\mu^2|v|_X^2}{\gamma}+ c\gamma \mu^2|v|_X^2 +c\gamma \|f\|^2+c\mu(\mathcal{R}_0^2\mathcal{R}_1)^2+c\mu \|f\|^2+c\mu^3|v|_X^2,
\end{align*}
$$K_7:= c(\mathcal{R}_0\mathcal{R}_1^3+\|f\|^2+\mu^2|v|_X^2),$$
$$K_5:= \frac{c(\mathcal{R}_0^{\frac{1}{2}}\mathcal{R}_1^2 )^4}{\gamma^3}+\frac{c\{\mathcal{R}_0\mathcal{R}_1[\mathcal{R}_0^2\mathcal{R}_1+ (\gamma+\mu) \mathcal{R}_0+ \mu   |v|_X+\|f\|]\}^2}{\gamma}.$$
\subsection{Time derivative bound}
We realize from $\eqref{Galerkin}$ that we have
\begin{align*}
\|w'(s)\|\leq \mathcal{R}_2+ c\mathcal{R}_0^2\mathcal{R}_1+ (\gamma+\mu)\mathcal{R}_0+\|f\|+\mu|v|_X.
\end{align*}
Thus 
\begin{align*}
\sup_{s\geq -k}\|w'(s)\|\leq \mathcal{R'},
\end{align*}
where
\begin{align}
\mathcal{R'}:= \mathcal{R}_2+ c\mathcal{R}_0^2\mathcal{R}_1+ (\gamma+\mu)\mathcal{R}_0+\|f\|+\mu|v|_X. \label{K11}
\end{align}
\subsection{Passing to the limit}
To summarize our $L^2$, $H^1$, $H^2$ and the time derivative bounds, we have
\begin{equation}\label{bounds}
\begin{aligned}
&\sup_{s\geq -k}\|w_n(s)\|\leq \mathcal{R}_0=O(\mu^0), \\
&\sup_{s\geq -k}\|w_n(s)\|_{H^1}\leq \mathcal{R}_1=O(\mu^{\frac{1}{2}}), \\
&\sup_{s\geq -k}\|w_n(s)\|_{H^2}\leq \mathcal{R}_2=O(\mu^2),\\
&\sup_{s\geq -k}\|w'_n(s)\|\leq \mathcal{R'}=O(\mu^2),
\end{aligned}
\end{equation}
as $\mu \to \infty$, where $\mathcal{R}_0, \mathcal{R}_1$, $\mathcal{R}_2$ and $\mathcal{R'}$ are independent of $k$ and $n$, and defined in $\eqref{K1}$, $\eqref{K6}$, $\eqref{K10}$ and $\eqref{K11}$, respectively, and where $w_n$ is the solution of the initial value problem \eqref{Galerkin}-\eqref{ic}. Thus we have a bounded solution $w_n$  to the Galerkin approximation $\eqref{Galerkin}$ of  the equation $\eqref{main}$ with initial condition $\eqref{ic}$, on the interval $[-k, \infty)$, and satisfying $\eqref{bounds}$; we will call it $w_{n,k}$ to emphasize the initial time $-k$, and consider
\begin{align*}
w_{n,k}\in C_b([-k, \infty); H^2)\cap C_b^1([-k, \infty); L^2).
\end{align*}
We now focus on the interval $[-1,1]$. Since $H_n$, defined in $\eqref{Hm}$, is finite-dimensional, we may invoke the Arzela-Ascoli compactness theorem to extract a subsequence of $w_{n,k}$, denoted by $w_{n,k}^{(1)}$, such that $w_{n,k}^{(1)}\rightarrow w_n^{(1)}$, as $k\rightarrow \infty$, where $w_n^{(1)}$ is a bounded solution of the Galerkin approximation $\eqref{Galerkin}$ on the interval $[-1,1]$. Let $j\in \mathbb{N}$, we will use an induction iterative procedure to define $w_{n,k}^{(j+1)}$ to be subsequence of $w_{n,k}^{(j)}$, all of which are subsequences of $w_{n,k}$. Indeed, we have already defined $w_{n,k}^{(1)}$. Suppose $w_{n,k}^{(j)}$ is defined, and is a subsequence of $w_{n,k}$. We apply again the Arzela-Ascoli compactness theorem to extract a subsequence of $w_{n,k}^{(j)}$, denoted by $w_{n,k}^{(j+1)}$, such that $w_{n,k}^{(j+1)}\rightarrow w_n^{(j+1)}$, as $k\rightarrow \infty$, uniformly on $[-(j+1),(j+1)]$, where $w_{n}^{(j+1)}$ is a bounded solution of the Galerkin approximation \eqref{Galerkin} on the interval $[-(j+1),(j+1)]$. Notice that $w_{n}^{(j)}$ satisfies all the estimates in $\eqref{bounds}$ in the interval $[-j,j]$. By the Cantor diagonal process we have that $w_{n,k}^{(k)} \to w_n$, where $w_n$ is a bounded solution of the Galerkin approximation $\eqref{Galerkin}$ on all of $\mathbb{R}$ satisfying all the estimates above. Thanks to the compact embeddings
\begin{align*}
H^2\hookrightarrow H^1 \hookrightarrow L^2,
\end{align*}
and 
\begin{align*}
\sup_{s\in \mathbb{R}}\|w_{n}\|_{H^2}\leq \mathcal{R}_2 \text{  and  } \sup_{s\in \mathbb{R}}\|w'_{n}\|\leq \mathcal{R'}, 
\end{align*}
we can apply Aubin's compactness theorem (see, e.g., \cite{ConFoias} and \cite{TemamDyn}). For every $m\in \mathbb{N}$ there exist a subsequence $w_{n}^{(m)}$ of $w_n$ such that $w_{n}^{(m)}\rightarrow w^{(m)}$ on the interval $[-m,m]$, where $w^{(m)}$ is a bounded solution of $\eqref{main}$ on the interval $[-m,m]$.  Also,$w^{(m)}$ and $\frac{d}{ds}w^{(m)}$ satisfy estimates $\eqref{bounds}$ on the interval $[-m,m]$. Again by the Cantor diagonal process we have a subsequence $w_n^{(n)} \to w$ where $w$ is a bounded solution of $\eqref{main}$ on all of $\mathbb{R}$. Since $w$ and $w'$ also satisfy $\eqref{bounds}$ for all $s\in \mathbb{R}$, we have the following theorem: 
\newtheorem{existence}{Proposition}[section]
\begin{existence}
Let $v\in X$, where X is defined as in $\eqref{X}$. Then there exists a bounded solution $w\in Y$ of $\eqref{main}$, where $Y$ is defined as in $\eqref{Y}$. 
\end{existence}
Note that conditions ($\ref{E1}$) and ($\ref{E2}$) are only needed to get sharper bounds. Even without these conditions, there exists a bounded solution of $\eqref{main}$. But to have the bounds $\eqref{bounds}$, we need conditions ($\ref{E1}$) and ($\ref{E2}$). 
\newtheorem{cor1}[existence]{Theorem}
\begin{cor1}
Let $w$ be any bounded solution of $\eqref{main}$,  on $\mathbb{R}$, for some $v\in X$. Assume that $m$ is large enough such that both conditions \ref{E1} and \ref{E2} hold i.e,  
\begin{align}
\max\{\frac{\tilde{\tilde{\mathcal{R}}}_1L}{2\pi (m+1)}, \frac{\tilde{\mathcal{R}}_2L}{2\pi (m+1)}\}\leq 1\label{mcondition1},
\end{align}
where $\tilde{\tilde{\mathcal{R}}}_1$, $\tilde{\mathcal{R}}_2$ are defined as in $\eqref{K61}$, $\eqref{K71}$, respectively. Then $w$ satisfies the bounds in $\eqref{bounds}$.
\end{cor1}
\begin{proof} Given that $w$ is a bounded solution of $\eqref{main}$, we can mimic section  4. We integrate the evolution inequalities in that section from $s_0$ to $s$, then take $s_0$ to $-\infty$, to obtain the same bounds. 
\end{proof}
\newtheorem{rk}[existence]{Remark}
\begin{rk}
Assume that $v\in X$ is given.
\begin{enumerate}
\item[ ]
\item $v$ is independent of $\mu$ and $\gamma$. All the estimates we have depend on $|v|_X$. 
\item Observe that $\tilde{\tilde{\mathcal{R}}}_1=O(\mu^\frac{5}{2})$ and $\tilde{\mathcal{R}}_2=O(\mu^4)$ as $\mu \to \infty$. Therefore, condition $\eqref{mcondition1}$ implies that $m\geq O(\mu^4)$. 
\item We note the $\gamma$ dependence on these constant is of the form $\mathcal{R}_0=O(\gamma^{-1}), \mathcal{R}_1=O(\gamma^{-\frac{7}{2}}), \mathcal{R}_2=O(\gamma^{-17}), \text{ and } \mathcal{R'}=O(\gamma^{-17}) \text{ as } \gamma \to 0.$
\item Note that since $\tilde{\tilde{\mathcal{R}}}_1=O(\gamma^{-4})$ and $\tilde{\mathcal{R}}_2=O(\gamma^{-19.5})$ as $\gamma \to 0$, then condition $\eqref{mcondition1}$ implies that $m\geq O(\gamma^{-19.5})$. 
\end{enumerate}
\end{rk}

\section{Main Results}
\newtheorem{w=u}{Proposition}[section] 
\begin{w=u}
Let $u$ be a trajectory on the global attractor of the damped, driven NLS 
\begin{align}
iu_t+u_{xx}+|u|^2u+i\gamma u=f \label{NLS},
\end{align}
and let $w$ be a bounded solution of the equation $\eqref{main}$ with $v=P_mu$. Assume that $\mu$ is large enough so that
\begin{align}
\mathcal{R}_\infty^\frac{1}{2}(\mathcal{R}_\infty^0)^\frac{1}{2}<\mu, \label{mucondition}
\end{align}
holds, where $\mathcal{R}_\infty=c\mathcal{R}_0\mathcal{R}_1$, and $\mathcal{R}_\infty^0=\mathcal{R}_\infty |_{\mu=0}$. In addition, assume that $m$ is large enough such that both $\eqref{mcondition1}$ and 
\begin{align}
\frac{cL^2K_9}{\gamma^2 (m+1)^2}\leq 1 \label{mcondition2}
\end{align}
hold, where $K_9$ is defined in $\eqref{K13}$, and c is a universal constant. Then we have $w\equiv u$.
\end{w=u}
\newtheorem{rk5}[w=u]{Remark}
\begin{rk5}
Note that condition $\eqref{mucondition}$ can be achieved since $\mathcal{R}_\infty^\frac{1}{2}(\mathcal{R}_\infty^0)^\frac{1}{2}=O(\mu^\frac{1}{4})$ as $\mu \to \infty$. We realize that 
$$\frac{cL^2K_9}{2\gamma^2}\leq O(\gamma^{-\frac{383}{12}}),\qquad \frac{cL^2K_9}{2\gamma^2} \leq O(\mu^\frac{35}{12}),$$
as $\gamma \to 0$ and $\mu \to \infty$. Condition $\eqref{mcondition2}$ implies that we need to choose $m$ large enough such that $m\geq O(\gamma^{-\frac{383}{24}})$ and $m\geq O(\mu^\frac{35}{24})$. We note that these conditions are already achieved by $\eqref{mucondition}$, $\eqref{mcondition1}$ and Remark 4.3, up to a constant. 
\end{rk5}
Now we give a proof for Proposition 5.1. 
\begin{proof}
We first mention that all of the bounds that we obtained in section  4 also hold for  the solution of equation $\eqref{NLS}$, with $\mu=0$. Our notation for the time derivative bound and the square of the $L^\infty$ bound for  the solution $u$ will be $\mathcal{R'}^0$ and $\mathcal{R}_\infty^0$ (see \eqref{Kinfty} for $\mathcal{R}_\infty$), respectively. We will use the superscript 0 in $K_j^0$ to denote the constant $K_j$, but with $\mu=0$ in its formula. Taking the difference of the following equations
\begin{align*}
&iw_s+w_{xx}+|w|^2w+i\gamma w=f- i\mu [P_m(w)-u], \\
&iu_s+u_{xx}+|u|^2u+i\gamma u=f,
\end{align*}
we get
\begin{align*}
i\delta_s+\delta_{xx}+|w|^2w-|u|^2u+i\gamma \delta= -i\mu P_m\delta,
\end{align*}\\\
where $\delta:=w-u$. Note that 
\begin{align*}
|w|^2w-|u|^2u&= |\delta|^2\delta+ w\bar{u}\delta+\bar{w}u\delta+wu\bar{\delta}\\
&=|\delta|^2\delta+ 2Re(w\bar{u})\delta+wu\bar{\delta},
\end{align*}\\\
and hence 
\begin{equation}
i\delta_s+\delta_{xx}+|\delta|^2\delta+ 2Re(w\bar{u})\delta+wu\bar{\delta}+i\gamma \delta= -i\mu P_m \delta. \label{attractordelta}
\end{equation}\\\
Multiply $\eqref{attractordelta}$ by $\bar{\delta}$, integrate, and take the real parts, to get
\begin{align}
Im\int_0^L \delta \bar{\delta}_s=\|\delta_x\|^2- \int_0^L|\delta|^4-2\int_0^L Re(w\bar{u})|\delta|^2- Re\int_0^L wu\bar{\delta}^2. \label{imdeltadeltas}
\end{align}
Define $\Phi(s)$ as follows
\begin{align}
\Phi(s)= &\|\delta_x\|^2- \frac{1}{2}\int_0^L|\delta|^4-2\int_0^L Re(w\bar{u})|\delta|^2- Re\int_0^L wu\bar{\delta}^2.  \label{T}
\end{align}
Thus from $\eqref{imdeltadeltas}$, we have 
$$Im\int_0^L \delta \bar{\delta}_s=\Phi(s)- \frac{1}{2}\int_0^L|\delta|^4.$$
Now multiply $\eqref{attractordelta}$ by $P_m\bar{\delta}$, integrate, and take the real parts, to get
\begin{align}
Im\int_0^L P_m\delta P_m\bar{\delta}_s=&\|P_m\delta_x\|^2- Re\int_0^L |\delta|^2\delta P_m\bar{\delta}\notag\\
&-2\int_0^L Re(w\bar{u})Re(\delta P_m\bar{\delta})-Re\int_0^L wu\bar{\delta}P_m\bar{\delta}. \label{imdeltamdeltams}
\end{align}
Multiply $\eqref{attractordelta}$ by $\bar{\delta}_s$, integrate, and take the real parts, to get
\begin{align*}
\frac{d}{ds}\|\delta_x\|^2-\frac{1}{2}\frac{d}{ds}\int_0^L|\delta|^4=&-2\gamma Im\int_0^L \delta \bar{\delta}_s-2\mu Im\int_0^L P_m\delta P_m\bar{\delta}_s \\
&+4\int_0^L Re(w\bar{u})Re(\delta \bar{\delta}_s) +2Re\int_0^L wu\bar{\delta} \bar{\delta}_s.
\end{align*}
We then realize that,
\begin{align*}
\frac{d}{ds}\Phi(s)+2\gamma \Phi(s)= &-2\mu Im\int_0^L P_m\delta P_m\bar{\delta}_s-2\int_0^LRe(w\bar{u}_s)|\delta|^2 \\
&-Re\int_0^L(wu)_s\bar{\delta}^2+\gamma \int_0^L|\delta|^4. 
\end{align*}
Use $\eqref{imdeltamdeltams}$ above to get 
\begin{align*}
\frac{d}{ds}\Phi(s)+2\gamma \Phi(s)=&-2\mu\|P_m\delta_x\|^2+2\mu Re\int_0^L |\delta|^2\delta P_m\bar{\delta}\notag\\
&+4\mu \int_0^L Re(w\bar{u})Re(\delta P_m\bar{\delta})+2\mu Re\int_0^L wu\bar{\delta}P_m\bar{\delta}\notag \\
&-2\int_0^LRe(w\bar{u}_s)|\delta|^2-Re\int_0^L(wu)_s\bar{\delta}^2 \notag \\
&+\gamma \int_0^L|\delta|^4.
\end{align*}
Since condition $\eqref{mcondition1}$ is an assumption of the proposition, we may use the bounds we obtained in section  4. By using Agmon's inequality, along with derivative bound $\eqref{K11}$
and
\begin{align}
\|w\|_\infty^2 \leq c\|w\|\|w\|_{H^1}\leq c(\mathcal{R}_0\mathcal{R}_1):= \mathcal{R}_\infty=O(\mu^\frac{1}{2}), \label{Kinfty} 
\end{align}
as $\mu \to \infty$, so we have
\begin{align*}
 -2\int_0^L Re(w\bar{u}_s)|\delta|^2&\leq 2\|w\|_\infty \|\delta\|_\infty \|u_s\|\|\delta\| \notag \\
 &\leq c\mathcal{R}_\infty^\frac{1}{2}(\|\delta\|^{\frac{1}{2}}\|\delta\|_{H^1}^{\frac{1}{2}})\|u_s\|\|\delta\| \notag \\
 &\leq c\mathcal{R}_\infty^\frac{1}{2}\mathcal{R'}^0(\|\delta\|^{\frac{3}{2}}\|\delta\|_{H^1}^{\frac{1}{2}})  \notag \\
 &\leq c\mathcal{R}_\infty^\frac{1}{2}\mathcal{R'}^0\|\delta\|^2\notag \\
 &+c\mathcal{R}_\infty^\frac{1}{2}\mathcal{R'}^0\|\delta\|^{\frac{3}{2}}\|\delta_x\|^{\frac{1}{2}}\notag \\
 &\leq c[\mathcal{R}_\infty^\frac{1}{2}\mathcal{R'}^0+\frac{1}{\gamma^\frac{1}{3}}\mathcal{R}_\infty^\frac{2}{3}(\mathcal{R'}^0)^\frac{4}{3}]\|\delta\|^2 +\frac{\gamma}{3} \|\delta_x\|^2.
\end{align*}
Similar analysis can be done for the term $-Re\int_0^L (wu)_s\bar{\delta}^2$. Since
\begin{align*}
-Re\int_0^L (wu)_s\bar{\delta}^2= -Re\int_0^L w_su\bar{\delta}^2-Re\int_0^L wu_s\bar{\delta}^2, 
\end{align*}
We have 
\begin{align*}
-Re\int_0^L (wu)_s\bar{\delta}^2\leq & c[(\mathcal{R}_\infty^0)^\frac{1}{2}\mathcal{R'}+\frac{1}{\gamma^\frac{1}{3}}(\mathcal{R}_\infty^0)^\frac{2}{3}(\mathcal{R'})^\frac{4}{3}]\|\delta\|^2 +\frac{\gamma}{3} \|\delta_x\|^2 \\
&+ c[\mathcal{R}_\infty^\frac{1}{2}\mathcal{R'}^0+\frac{1}{\gamma^\frac{1}{3}}\mathcal{R}_\infty^\frac{2}{3}(\mathcal{R'}^0)^\frac{4}{3}]\|\delta\|^2 +\frac{\gamma}{3} \|\delta_x\|^2.
\end{align*}
We also have
\begin{align*}
2\mu Re\int_0^L |\delta|^2\delta P_m\bar{\delta}&\leq 2\mu \|\delta\|_\infty^2\|\delta\|\|P_m\delta\|\notag \\
&\leq 4\mu (\mathcal{R}_\infty+ \mathcal{R}_\infty^0) \|\delta\|\|P_m\delta\| \notag \\
&\leq  4\mu (\mathcal{R}_\infty+ \mathcal{R}_\infty^0) \|\delta\|^2.
\end{align*}
After similar treatment of the terms $4\mu \int_0^L Re(w\bar{u})Re(\delta P_m\bar{\delta})$, \\$2\mu Re\int_0^L wu\bar{\delta}P_m\bar{\delta}$ and $\gamma \int_0^L|\delta|^4$, we obtain 
\begin{align*}
4\mu \int_0^L Re(w\bar{u})Re(\delta P_m\bar{\delta})\leq 4\mu (\mathcal{R}_\infty)^\frac{1}{2} (\mathcal{R}_\infty^0)^\frac{1}{2}\|\delta\|^2,
\end{align*}
\begin{align*}
2\mu Re\int_0^L wu\bar{\delta}P_m\bar{\delta}\leq 2\mu (\mathcal{R}_\infty)^\frac{1}{2} (\mathcal{R}_\infty^0)^\frac{1}{2}\|\delta\|^2,
\end{align*}
\begin{align*}
\gamma \int_0^L|\delta|^4\leq \gamma \|\delta\|_\infty^2\|\delta\|^2\leq 2\gamma(\mathcal{R}_\infty+\mathcal{R}_\infty^0)\|\delta^2\|. 
\end{align*}
Combine the above terms, to obtain
\begin{align}
\frac{d}{ds}\Phi(s)+2\gamma \Phi(s)+2\mu \|P_m\delta_x\|^2\leq cK_8 \|\delta\|^2+ \gamma \|\delta_x\|^2, \label{Tinequality1}
\end{align}
where
\begin{align*}
K_8:=&\mu (\mathcal{R}_\infty+\mathcal{R}_\infty^0)+\mu ((\mathcal{R}_\infty)^\frac{1}{2} (\mathcal{R}_\infty^0)^\frac{1}{2})+(\mathcal{R}_\infty^0)^\frac{1}{2}\mathcal{R'}+\frac{1}{\gamma^\frac{1}{3}}(\mathcal{R}_\infty^0)^\frac{2}{3}(\mathcal{R'})^\frac{4}{3} \\
&+\mathcal{R}_\infty^\frac{1}{2}\mathcal{R'}^0+\frac{1}{\gamma^\frac{1}{3}}\mathcal{R}_\infty^\frac{2}{3}(\mathcal{R'}^0)^\frac{4}{3}. 
\end{align*}
We realize that $K_8=O(\mu^\frac{8}{3})$ as $\mu \to \infty$. Also, from $\eqref{T}$
\begin{align*}
\Phi(s)&\geq \|\delta_x\|^2- c(\mathcal{R}_\infty+\mathcal{R}_\infty^0+\mathcal{R}_\infty^\frac{1}{2}(\mathcal{R}_\infty^0)^\frac{1}{2}) \|\delta\|^2, 
\end{align*}
and hence 
\begin{align}
\|\delta_x\|^2\leq \Phi(s)+ c(\mathcal{R}_\infty+\mathcal{R}_\infty^0+\mathcal{R}_\infty^\frac{1}{2}(\mathcal{R}_\infty^0)^\frac{1}{2}) \|\delta\|^2.  \label{Tinequality2}
\end{align}
Using $\eqref{Tinequality1}$ and $\eqref{Tinequality2}$, we conclude that 
\begin{align*}
\frac{d}{ds}\Phi(s)+2\gamma \Phi(s)\leq c[K_8+\gamma (\mathcal{R}_\infty+\mathcal{R}_\infty^0+\mathcal{R}_\infty^\frac{1}{2}(\mathcal{R}_\infty^0)^\frac{1}{2})]\|\delta\|^2+ \gamma \Phi(s),
\end{align*}
so
\begin{align*}
\frac{d}{ds}\Phi(s)+\gamma \Phi(s)\leq c[K_8+\gamma (\mathcal{R}_\infty+\mathcal{R}_\infty^0+\mathcal{R}_\infty^\frac{1}{2}(\mathcal{R}_\infty^0)^\frac{1}{2})]\|\delta\|^2.
\end{align*}
So we have 
$$\Phi(s)\leq \frac{c[K_8+\gamma (\mathcal{R}_\infty+\mathcal{R}_\infty^0+\mathcal{R}_\infty^\frac{1}{2}(\mathcal{R}_\infty^0)^\frac{1}{2})]}{\gamma}\sup_{s\in\mathbb{R}}\|\delta(s)\|^2.$$
Thus, 
\begin{align}
\|\delta_x\|^2\leq  \frac{c[K_8+\gamma (\mathcal{R}_\infty+\mathcal{R}_\infty^0+\mathcal{R}_\infty^\frac{1}{2}(\mathcal{R}_\infty^0)^\frac{1}{2})]}{\gamma}\sup_{s\in\mathbb{R}}\|\delta(s)\|^2. \label{inverseP}
\end{align}
The inequality $\eqref{inverseP}$ is a 'reverse` Poincar\'e type inequality. From $\eqref{l2inequality}$,
\begin{align*}
\frac{d}{ds}\|\delta\|^2&+2\gamma \|\delta\|^2+2\mu \|P_m\delta\|^2\leq 2\|w\|_\infty \|u\|_\infty \|\delta\|^2\\
&\leq 2\mathcal{R}_\infty^\frac{1}{2}(\mathcal{R}_\infty^0)^\frac{1}{2} \|\delta\|^2\\
&= 2\mathcal{R}_\infty^\frac{1}{2}(\mathcal{R}_\infty^0)^\frac{1}{2} \|P_m\delta\|^2+2\mathcal{R}_\infty^\frac{1}{2}(\mathcal{R}_\infty^0)^\frac{1}{2} \|Q_m\delta\|^2\\
&\leq 2\mathcal{R}_\infty^\frac{1}{2}(\mathcal{R}_\infty^0)^\frac{1}{2} \|P_m\delta\|^2+\frac{2\mathcal{R}_\infty^\frac{1}{2}(\mathcal{R}_\infty^0)^\frac{1}{2}}{((m+1)\frac{2\pi}{L})^2} \|\delta_x\|^2 \\
&\leq 2\mathcal{R}_\infty^\frac{1}{2}(\mathcal{R}_\infty^0)^\frac{1}{2}\|P_m\delta\|^2+ \frac{cL^2K_9}{\gamma (m+1)^2}\sup_{s\in\mathbb{R}}\|\delta(s)\|^2,
\end{align*}
where 
\begin{align}
K_9:=\mathcal{R}_\infty^\frac{1}{2}(\mathcal{R}_\infty^0)^\frac{1}{2}[K_8+\gamma (\mathcal{R}_\infty+\mathcal{R}_\infty^0+\mathcal{R}_\infty^\frac{1}{2}(\mathcal{R}_\infty^0)^\frac{1}{2})]. \label{K13}
\end{align}
Thus, if we choose $\mu$ large enough so that $c\mathcal{R}_\infty^\frac{1}{2}(\mathcal{R}_\infty^0)^\frac{1}{2}=O(\mu^\frac{1}{4})\leq 2\mu$ (which is the condition $\eqref{mucondition}$), we get
\begin{align*}
\frac{d}{ds}\|\delta\|^2+2\gamma \|\delta\|^2\leq\frac{cL^2K_9}{\gamma (m+1)^2}\sup_{s\in\mathbb{R}}\|\delta(s)\|^2,
\end{align*}
and hence 
\begin{align*}
\sup_{s\in\mathbb{R}}\|\delta(s)\|^2\leq \frac{cL^2K_9}{2\gamma^2 (m+1)^2}\sup_{s\in\mathbb{R}}\|\delta(s)\|^2.
\end{align*}
Since $m$ is chosen to be large enough satisfying condition $\eqref{mcondition2}$, we conclude that $\sup_{s\in\mathbb{R}}\|\delta(s)\|^2=0$. Thus $\delta\equiv0$. This implies that $w\equiv u$. 
\end{proof}
\newtheorem{W1}[w=u]{Theorem} \label{lipschitz}
\begin{W1}
Let $v\in B_\rho:=\{v\in X; |v|_X\leq \rho \}$ for some positive $\rho$. Assume that 
\begin{align}
c(\mathcal{R}_0\mathcal{R}_1)<\mu, \label{mucondition2}
\end{align}
holds, and for such $\mu$, conditions $\eqref{mcondition1}$ and 
\begin{align}
\frac{cL^2K_{10}}{\gamma^2 (m+1)^2}\leq 1, \label{mcondition3}
\end{align}
hold, where $K_{10}$ is defined in \eqref{K14}. Then the map $W: X\rightarrow Y$, where $W(v):=w$ is a bounded solution of $\eqref{main}$ provided by Proposition 4.1, is well-defined, and $P_mW: X\rightarrow X$ is a locally Lipschitz function with Lipschitz constant $L_W(\rho)$ given in $\eqref{LW}$. 
\end{W1}
\newtheorem{rknew}[w=u]{Remark}
\begin{rknew} 
\begin{enumerate} 
\item[ ]
\item Recall that $\mathcal{R}_0$, $\mathcal{R}_1$, $\mathcal{R'}$, $\mathcal{R}_\infty$ depend on $|v|_X$ which is controlled by $\rho$. Also note that condition $\eqref{mucondition2}$ can be achieved since $c(\mathcal{R}_0\mathcal{R}_1)=O(\mu^\frac{1}{2})$ as $\mu \to \infty$.
\item Notice that $L_W(\rho)=O(\mu^\frac{3}{2})$ as $\mu \to \infty$, $L_W(\rho)=O(\gamma^{-\frac{9}{2}})$ as $\gamma \to 0$ and $L_W(\rho)$ can be bounded independent of $m$.
\item We note that 
$$\frac{cL^2K_{10}}{2\gamma^2}\leq O(\gamma^{-\frac{65}{2}}),\qquad \frac{cL^2K_{10}}{2\gamma^2} \leq O(\mu^\frac{7}{2}),$$
as $\gamma \to 0$ and $\mu \to \infty$. Condition $\eqref{mcondition3}$ implies that we need to choose $m$ large enough such that $m\geq O(\gamma^{-\frac{65}{4}})$ and $m\geq O(\mu^\frac{7}{4})$. These conditions are already achieved by $\eqref{mucondition2}$, $\eqref{mcondition1}$ and Remark 4.3, up to a constant. 
\end{enumerate}
\end{rknew}
Now we give the proof of Theorem 5.3. 
\begin{proof}
Note that all constants $\mathcal{R}_0, \mathcal{R}_1, \mathcal{R}_2, \mathcal{R'}$ and $\mathcal{R}_\infty$ depend on $|v|_X$. So, since we are in a ball $B_\rho \subset X$, all of these constants will depend on $\rho$. Let $v, \tilde{v}\in B_\rho$ such that $W(v)=w$ and $W(\tilde{v})=\tilde{w}$.  Since $w$ and $\tilde{w}$ are the solutions of the equation $\eqref{main}$ for $v$ and $\tilde{v}$ , respectively , the following hold:
$$iw_s+w_{xx}+|w|^2w+i\gamma w=f- i\mu [P_m(w)-v],$$
$$i\tilde{w}_s+\tilde{w}_{xx}+|\tilde{w}|^2\tilde{w}+i\gamma \tilde{w}=f- i\mu [P_m(\tilde{w})-\tilde{v}].$$
Subtract, denoting $\delta:=w-\tilde{w}$ and $\eta:=v-\tilde{v}$, to obtain 
\begin{equation}
i\delta_s+\delta_{xx}+|\delta|^2\delta+2Re(w\bar{\tilde{w}})\delta+w\tilde{w}\bar{\delta}+i\gamma \delta+i\mu P_m\delta= i\mu \eta \label{deltaeta}.
\end{equation}\\\\
Multiply $\eqref{deltaeta}$ by $\bar{\delta}$, integrate, and take the real parts, to get
\begin{align*}
Im\int_0^L \delta \bar{\delta}_s=&\|\delta_x\|^2- \int_0^L |\delta|^4-2\int_0^L Re(w\bar{\tilde{w}})|\delta|^2\notag \\
&- Re\int_0^L w\tilde{w}\bar{\delta}^2-\mu Im\int_0^L \eta\bar{\delta}.
\end{align*}
Define $\Psi(s)$ as follows
\begin{align}
\Psi(s)= &\|\delta_x\|^2- \frac{1}{2}\int_0^L |\delta|^4-2\int_0^L Re(w\bar{\tilde{w}})|\delta|^2 \notag \\
&- Re\int_0^L w\tilde{w}\bar{\delta}^2-\mu Im\int_0^L \eta\bar{\delta}.  \label{Psi}
\end{align}
Now multiply $\eqref{deltaeta}$ by $P_m\bar{\delta}$, integrate, and take the real parts, to get

\begin{align}
Im\int_0^L P_m\delta P_m\bar{\delta}_s=&\|P_m\delta_x\|^2- Re\int_0^L |\delta|^2\delta P_m\bar{\delta} \notag \\
&-2\int_0^L Re(w\bar{\tilde{w}})Re(\delta P_m\bar{\delta})- Re\int_0^L w\tilde{w}\bar{\delta}P_m\bar{\delta} \notag \\
&-\mu Im\int_0^L \eta P_m\bar{\delta}. \label{imdeltamdeltamseta}
\end{align}

Multiply $\eqref{deltaeta}$ by $\bar{\delta}_s$, integrate, and take the real parts, to get
\begin{align*}
\frac{d}{ds}\|\delta_x\|^2-\frac{1}{2}\frac{d}{ds}\int_0^L|\delta|^4=&-2\gamma Im\int_0^L \delta \bar{\delta}_s-2\mu Im\int_0^L P_m\delta P_m\bar{\delta}_s \notag \\
&+4\int_0^L Re(w\bar{\tilde{w}})Re(\delta \bar{\delta}_s) +2Re\int_0^L w\tilde{w}\bar{\delta} \bar{\delta}_s\\
&+ \mu Im\int_0^L \eta \bar{\delta}_s.
\end{align*}
We then realize that
\begin{align*}
\frac{d}{ds}\Psi(s)+2\gamma \Psi(s)= &-2\mu Im\int_0^L P_m\delta P_m\bar{\delta}_s-2\int_0^LRe(w\bar{\tilde{w}})_s|\delta|^2 \\
&-Re\int_0^L(w\tilde{w})_s\bar{\delta}^2+\gamma \int_0^L|\delta|^4 \\ 
&-\mu Im\int_0^L \eta_s\bar{\delta}-2\gamma \mu Im\int_0^L \eta\bar{\delta}. 
\end{align*}
Use $\eqref{imdeltamdeltamseta}$ above to get
\begin{align*}
\frac{d}{ds}\Psi(s)+2\gamma \Psi(s)=&-2\mu\|P_m\delta_x\|^2+2\mu Re\int_0^L |\delta|^2\delta P_m\bar{\delta}\notag\\
&+4\mu \int_0^L Re(w\bar{u})Re(\delta P_m\bar{\delta})\\
&+2\mu Re\int_0^L wu\bar{\delta}P_m\bar{\delta}-2\int_0^LRe(w\bar{u}_s)|\delta|^2\\
&-Re\int_0^L(wu)_s\bar{\delta}^2+\gamma \int_0^L|\delta|^4\\
&-\mu Im\int_0^L \eta_s\bar{\delta}-2\gamma \mu Im\int_0^L \eta \bar{\delta}.
\end{align*} 
We estimate as before to obtain
\begin{align*}
\|\delta_x\|^2\leq \frac{(c(\mu+\gamma) \mathcal{R}_\infty+c\gamma^{-\frac{1}{3}}\mathcal{R}_\infty^\frac{2}{3}\mathcal{R'}^\frac{4}{3})}{\gamma}\sup_{s\in\mathbb{R}}\|\delta(s)\|^2+(\mu+3\gamma \mu) |\eta|_X \sup_{s\in\mathbb{R}}\|\delta(s)\| .
\end{align*}
Mulitply $\eqref{deltaeta}$ with $\bar{\delta}$, integrate, and take the imaginary parts to obtain 
$$\frac{d}{ds}\|\delta\|^2+2\gamma \|\delta\|^2+2\mu \|P_m\delta\|^2=2\mu Re\int_0^L \eta \bar{\delta}- 2 Im\int_0^L w\tilde{w}\bar{\delta}^2 .$$
We make similar estimates again, and take advantage of the condition $\eqref{mucondition2}$, to get  
\begin{align*}
\frac{d}{ds}\|\delta\|^2+2\gamma \|\delta\|^2\leq & \frac{cL^2K_{10}}{\gamma (m+1)^2}\sup_{s\in\mathbb{R}}\|\delta(s)\|^2\\ 
&+ \frac{c\mathcal{R}_\infty(\mu+3\gamma \mu)}{(\frac{2\pi}{L})^2(m+1)^2}|\eta|_X \sup_{s\in\mathbb{R}}\|\delta(s)\|\;,
\end{align*}
where 
\begin{align}
K_{10}=\mathcal{R}_\infty[(\mu+\gamma) \mathcal{R}_\infty+\gamma^{-\frac{1}{3}}\mathcal{R}_\infty^\frac{2}{3}\mathcal{R'}^\frac{4}{3}].\label{K14}
\end{align}
By $\eqref{mcondition3}$ we have  
\begin{align}
\sup_{s\in\mathbb{R}}\|\delta(s)\| \leq \frac{c\mathcal{R}_\infty(\mu+3\gamma \mu)}{(\frac{2\pi}{L})^2(m+1)^2 \gamma}|\eta|_X. \label{welldefined}
\end{align}
Note that $\eqref{welldefined}$ implies that the $W$-map is well-defined. Now, 
\begin{align}
\|P_m\delta_{xx}\|\leq m^2\|P_m\delta\|\leq m^2\|\delta\|\leq  \frac{c\mathcal{R}_\infty(\mu+3\gamma \mu)}{(\frac{2\pi}{L})^2 \gamma}|\eta|_X \label{lipl2}.
\end{align}
From $\eqref{deltaeta}$ and $\eqref{lipl2}$, 
\begin{align*}
\|P_m\delta_s\|&\leq \|P_m\delta_{xx}\|+\|P_m(|w|^2\delta+w\tilde{w}\bar{\delta}+|\tilde{w}|^2\delta)\|+(\gamma+\mu)\|P_m\delta\|+\mu\|\eta\| \notag \\
&\leq m^2 \|\delta\|+ c\mathcal{R}_\infty \|\delta\|+ (\gamma+\mu)\|\delta\|+\mu\|\eta\| \notag \\
&\leq ((m^2+ c\mathcal{R}_\infty+ \gamma+\mu)(\frac{c\mathcal{R}_\infty(\mu+3\gamma \mu)}{(\frac{2\pi}{L})^2(m+1)^2 \gamma})+\mu) |\eta|_X,
\end{align*}
so that
\begin{align*}
\|P_m\delta\|+\|P_m\delta_{s}\|\leq L_W |\eta|_X,
\end{align*}
where 
\begin{align}
L_W:=(m^2+ c\mathcal{R}_\infty+ \gamma+\mu+1)(\frac{c\mathcal{R}_\infty(\mu+3\gamma \mu)}{(\frac{2\pi}{L})^2(m+1)^2 \gamma})+\mu. \label{LW}
\end{align}
Thus, 
$$|P_m\delta |_X \leq L_W|\eta|_X,$$
i.e, 
$$|P_mW(v)-P_mW(\tilde{v})|_X\leq L_W |v-\tilde{v}|_X.$$
\end{proof}

\section{The determining form}
For every trajectory $u$ in the global attractor, $\mathcal{A}$, we have 
\begin{align*}
|u|_X\leq R, 
\end{align*}
where
\begin{align} 
R:= \mathcal{R}_0^0+\mathcal{R'}^0, \label{R}
\end{align}
with $\mathcal{R}_0^0=\mathcal{R}_0|_{\mu=0}$ and $\mathcal{R'}^0=\mathcal{R'}|_{\mu=0}$. Let $u^*$ be a steady state of the damped, driven NLS $\eqref{NLSmain}$. Adapting the suggestion given in \cite{Form2}, we propose the following  \emph{determining form} for the damped-driven NLS:
\begin{align}
\frac{dv}{dt}= -|v-P_mW(v)|_{X,0}^2 (v- P_mu^*), \label{newform}
\end{align}
where $|\cdot|_{X,0}$ is defined in $\eqref{X0}$. The specific conditions on $m$ and its dependence on $R$ to guarantee the existence of a Lipschitz map $P_mW(v)$ are stated in the Theorem 6.1 below. 
\newtheorem{Formthm}{Theorem}[section] 
\begin{Formthm}
Suppose that the conditions of Theorem 5.3 hold for $\rho=4R$, where $R$ is defined in $\eqref{R}$. 
\begin{enumerate}
\item{The vector field in the determining form $\eqref{newform}$ is a Lipschitz map from the ball $\mathcal{B}_X^\rho(0)= \{v\in X: |v|_X< \rho\}$ into $X$. Thus, $\eqref{newform}$ is actually an ODE, in the Banach space X, which has a short time existence and uniqueness for the initial data in $\mathcal{B}_X^\rho(0)= \{v\in X: |v|_X< \rho\}$. }
\item{The ball $\mathcal{B}_X^{3R}(P_mu^*)= \{v\in X: |v-P_mu^*|_X< 3R\} \subset \mathcal{B}_X^\rho(0)$ is forward invariant in time, under the dynamics of the determining form $\eqref{newform}$. Consequently, $\eqref{newform}$ has global existence and uniqueness for all initial data in $\mathcal{B}_X^{3R}(P_mu^*)$.}
\item{Every solution of  the determining form $\eqref{newform}$, with initial data in $\mathcal{B}_X^{3R}(P_mu^*)$, converges to a steady state of the determining form $\eqref{newform}$. }
\item{All the steady states of the determining form, $\eqref{newform}$, that are contained in the ball $\mathcal{B}_X^\rho(0)$ are given by the form $v(s)=P_mu(s)$, for all $s\in \mathbb{R}$, where $u(s)$ is a trajectory that lies on the global attractor, $\mathcal{A}$, of $\eqref{NLSmain}$.}
\end{enumerate}
\end{Formthm}
\begin{proof}
We use the fact that $P_mW$ is a locally Lipschitz map to prove item (1) above. For item (2) and (3), we use dissipative property of $\eqref{newform}$. To prove item (4), we realize that the right-hand side of $\eqref{newform}$ is zero when either $v=P_mu^*$ or $v=P_mw$. In either case, we show that $v(s)=P_mu(s)$, for all $s\in \mathbb{R}$, where $u(s)$ is a trajectory that lies on the global attractor, $\mathcal{A}$, of $\eqref{NLSmain}$. For details see \cite{Form2}. 
\end{proof}

\section{A new proof of the determining modes property}
Here we give the proof of Theorem 3.1:  
\begin{proof}
We assume $u(s)$ and $\tilde{u}(s)$ are trajectories on the global attractor, $\mathcal{A}$, of $\eqref{NLSmain}$, and  $P_m(u(s))=P_m(v(s))$ for all time $s\in \mathbb{R}$, and for some $m\in \mathbb{N}$ to be chosen later. Then, 
$$iu_s+u_{xx}+|u|^2u+i\gamma u=f,$$
$$i\tilde{u}_s+\tilde{u}_{xx}+|\tilde{u}|^2\tilde{u}+i\gamma \tilde{u}=f.$$
Subtract, denoting $\delta:=u-\tilde{u}$, to obtain 
\begin{equation}
i\delta_s+\delta_{xx}+|\delta|^2\delta+2Re(u\bar{\tilde{u}})\delta+u\tilde{u}\bar{\delta}+i\gamma \delta=0, \notag 
\end{equation}
which is precisely (5.4), but with $\mu=0$, $w$ replaced by $u$, and $\bar{u}$ replaced by $\bar{\tilde{u}}$. Following the proof of Theorem 5.1, we obtain the analog of (5.13) with $\mu=0$: 
\begin{align*}
\|\delta_x(s)\|\leq K_{11}\sup_{s\in \mathbb{R}}\|\delta(s)\|, 
\end{align*}
where
\begin{align}\label{K15}
K_{11}=\sqrt{\frac{(c\gamma \mathcal{R}_\infty^0+c\gamma^{-\frac{1}{3}}(\mathcal{R}_\infty^0)^\frac{2}{3}(\mathcal{R'}^0)^\frac{4}{3})}{\gamma}}\;.
\end{align}
Then, since $P_m\delta=0$, we have
 \begin{align*}
 \|Q_m\delta\|\leq \frac{L}{2\pi (m+1)}\|\delta_x\|&\leq \frac{L}{2\pi (m+1)}K_{11}\sup_{s\in \mathbb{R}}\|\delta\| \\
 &=\frac{L}{2\pi (m+1)}K_{11}\sup_{s\in \mathbb{R}}\|Q_m\delta\|.
 \end{align*}
 Thus, if we choose 
 \begin{align*}
 m\geq \frac{L}{2\pi}K_{11}-1,
 \end{align*}
 we obtain that $Q_m\delta=0$. As a result, $u(s)=\tilde{u}(s)$.
 \end{proof}
\newtheorem{Rk7}{Remark}[section]
\begin{Rk7} 
\begin{enumerate}
\item[ ]
\item By tracking the $\|f\|$ and $\gamma$ dependence of the bounds throughout the paper, we have that   $\mathcal{R}_0=O(\|f\|, \gamma^{-1}) \text{,   } \mathcal{R}_0^0=O(\|f\|, \gamma^{-1}), \mathcal{R}_1=O(\|f\|^3, \gamma^{-3.5}), \\\mathcal{R}_1^0=O(\|f\|^3, \gamma^{-3}), \mathcal{R}_2=O(\|f\|^{13}, \gamma^{-17}), \mathcal{R}_2^0=O(\|f\|^{13}, \gamma^{-15}), \\\mathcal{R'}=O(\|f\|^{13},\gamma^{-17}),\mathcal{R'}^0=O(\|f\|^{13}, \gamma^{-15}), \mathcal{R}_\infty=O(\|f\|^{4}, \gamma^{-4.5}),\\\mathcal{R}_\infty^0=O(\|f\|^{4}, \gamma^{-4}),  K_8=O(\|f\|^{20}, \gamma^{-\frac{77}{3}}), K_8^0=O(\|f\|^{20}, \gamma^{-23}),\\K_9=O(\|f\|^{24}, \gamma^{-\frac{359}{12}}), K_9^0=O(\|f\|^{24}, \gamma^{-27}), K_{10}=O(\|f\|^{24}, \gamma^{-\frac{61}{2}}), \\K_{10}^0=O(\|f\|^{24}, \gamma^{-27})$, as $\|f\| \to \infty$ and $\gamma \to 0$. 
\item Since $\mathcal{R'}^0=O(\|f\|^{13})$ and $\mathcal{R}_\infty^0=O(\|f\|^4)$, from $\eqref{K15}$ we have $K_{11}=O(\|f\|^{10})$ as $\|f\| \to \infty$. Thus, from $\eqref{numberofmodes}$, a sufficient number of determining modes is of order $m=O(\|f\|^{10})$. 
\item Similarly, since $\mathcal{R'}^0=O(\gamma^{-15})$ and $\mathcal{R}_\infty^0=O(\gamma^{-4})$, we have $K_{11}=O(\gamma^{-12})$ as $\gamma \to 0$. Thus a sufficient number of determining modes is of order $m=O(\gamma^{-12})$. 
\item  Following the analysis in the earlier work of Goubet \cite{Goubet}, one can show that a sufficient number of the determining modes is of order $O(\gamma^{-12.5})$ as $\gamma \to 0$ and $O(\|f\|^{12})$ as $\|f\| \to \infty$.   
\end{enumerate}
\end{Rk7}

\section{Acknowledgements}

E.S.T. would like to acknowledge the kind hospitality of the Instituto Nacional de Matem\' {a}tica  Pura  e Aplicada (IMPA), Brazil, where part of this work was completed. The work of M.J. is supported in part by National Science Foundation (NSF) Grant Numbers DMS-1008661 and DMS-1109638, and that of E.S.T.  is supported in part by the NSF grants DMS-1009950, DMS-1109640 and DMS-1109645, as well as by the CNPq-CsF grant \# 401615/2012-0, through the program Ci\^encia sem Fronteiras.\\\\


\begin{thebibliography}{100}
\bibitem{AzOT} A. Azouani, E. Olson and E. S. Titi, Continuos data assimilation using general interpolant observables,  \emph{Journal of Nonlinear Science}, Volume 24, Issue 2, p. 277-304 (2014).
\bibitem{AzT} A. Azouani and E. S. Titi, Feedback control of nonlinear dissipative systems by finite determining parameters- a reaction-diffusion paradigm, \emph{arXiv:1301.6992v1}
\bibitem{BEC} W. Bao and D. Jaksch, An explicit unconditionally stable numerical method for solving damped nonlinear Schr\"odingier equations with a focusing nonlinearity, \emph{SIAM J. Numer. Anal.}, Vol 4, No. 4, p. 1406-1426, (2003).
\bibitem{Physics} K. J. Blow and N. J. Doran, Global and Local Chaos in the Pumped Nonlinear Schr\"{o}dinger Equation, \emph{Physical Review Letters}, Vol 52, No 7, p.526-539 (1984)
\bibitem{Bourgain} J. Bourgain, Fourier transformation restriction phenomena for certain lattice subsets and applications to nonlinear evolution equations, Part I: Schr\"{o}dinger equations, \emph{GAFA}, 3, p. 107-156 (1993)
\bibitem{ConFoias} P. Constantin and C. Foias, \emph{Navier-Stokes Equations},  University of Chicago press, Chicago, IL, (1988). 
\bibitem{Fibich} G. Fibich, Self-focusing in the damped nonlinear Schr\"{o}dinger equation, \emph{SIAM J. Appl. Math}, Vol 61, No 5, p.1680
\bibitem{Form1} C. Foias, M. S. Jolly, R. Kravchenko and E. S. Titi,  A determining form for the 2D Navier-Stokes equations - the Fourier modes case, \emph{Journal of Mathematical Physics}, 53, p. 115623 (2012). 
\bibitem{Form2} C. Foias, M. S. Jolly, R. Kravchenko and E. S. Titi, A unified approach to determining forms for the 2D Navier-Stokes equations - The general interpolant case , \emph{Uspekhi Mat. Nauk}, Volume 69,  Issue 2(416), p. 177Ð200 (2014). \emph{arXiv:1309.0247}. 
\bibitem{FMRT} C. Foias, O. Manley, R. Rosa and R. Temam, Navier-Stokes Equations and Turbulence, \emph{Cambridge University Press}, (2001).  
\bibitem{modes} C. Foias, and G. Prodi, Sur le comportement global des solutions non-stationnaires des equations de Navier-Stokes en dimension 2, \emph{Rend. Sem. Mat. Univ. Padova.}, 39, p. 1-34, (1967). 
\bibitem{Gh} J-M. Ghidaglia, Finite dimensional behavior for weakly damped driven Schr\"{o}dinger Equations, \emph{Annales de l'I. H. P., section C}, tome 5, $n^0$ 4, p. 365-405, (1988).  
\bibitem{Goubet} O. Goubet, Regularity of the attractor for a weakly damped nonlinear Schr\"{o}dinger Equation, \emph{Appl. Anal.}, 60, p. 99-119, (1996).
\bibitem{Hale} J.K. Hale and G. Raugel, Regularity, determining modes and Galerkin methods, \emph{J. Math. Pures Appl.}, 82, p. 1075-1136, (2003). 
\bibitem{JoT} D. A. Jones and E. S. Titi, Upper bounds on the number of determining modes, nodes and volume elements for the Navier-Stokes equations, \emph{Indiana University Mathematics Journal}, Vol 42, No. 3, (1993). 
\bibitem{OliverTiti} M. Oliver and E. S. Titi, Analycity of the attractor and the number of determining nodes for a weakly damped driven nonlinear Schr\"{o}dinger equation, \emph{Indiana University Mathematics Journal}, Vol 47, p. 49-73, (1998). 
\bibitem{TemamDyn} R. Temam, \emph{Infinite-dimensional Dynamical Systems in Mechanics and Physics}, second ed., vol. 68 of Applied Mathematical Sciences. Springer-Verlag, New York, (1997).    
\bibitem{Wang} X. Wang, An energy equation for the weakly damped driven nonlinear Schr\"{o}dinger equations and its application to their attractors, \emph{Physica D}, 88, p 167-175, (1995).   
\end{thebibliography}
\end{document}